\documentclass[11pt,reqno]{amsart}
\usepackage[T1]{fontenc}
\usepackage[utf8]{inputenc}
\usepackage{lmodern,amsmath,amssymb,amsthm,mathtools,mathrsfs}
\usepackage[margin=1in]{geometry}
\usepackage{microtype,booktabs,array,enumitem}
\usepackage[hidelinks,unicode]{hyperref}
\setlist[enumerate]{leftmargin=2em,itemsep=3pt,topsep=4pt}
\allowdisplaybreaks[2]
\newtheorem{theorem}{Theorem}[section]
\newtheorem{maintheorem}{Theorem}

\newtheorem{proposition}[theorem]{Proposition}
\newtheorem{lemma}[theorem]{Lemma}
\newtheorem{corollary}[theorem]{Corollary}
\theoremstyle{definition}

\theoremstyle{remark}
\newtheorem{remark}[theorem]{Remark}
\newcommand{\A}{\mathbb A}\newcommand{\Gm}{\mathbb G_m}
\newcommand{\Q}{\mathbb Q}\newcommand{\R}{\mathbb R}\newcommand{\C}{\mathbb C}\newcommand{\Z}{\mathbb Z}\newcommand{\F}{\mathbb F}
\newcommand{\one}{\mathbbmone}\newcommand{\mathbbmone}{\mathbf 1}
\newcommand{\SH}{\mathrm{SH}}\newcommand{\DM}{\mathrm{DM}}
\newcommand{\CH}{\operatorname{CH}}\newcommand{\CW}{\widetilde{\operatorname{CH}}}
\newcommand{\GW}{\operatorname{GW}}\newcommand{\Witt}{\operatorname W}
\newcommand{\KM}{\mathbf K^{M}}\newcommand{\KMW}{\mathbf K^{MW}}
\newcommand{\Ish}{\mathbf I}\newcommand{\Wsh}{\mathbf W}
\newcommand{\Nis}{\mathrm{Nis}}\newcommand{\ret}{\mathrm{r\acute et}}
\newcommand{\gr}{\operatorname{gr}}\newcommand{\Spec}{\operatorname{Spec}}
\newcommand{\Hom}{\operatorname{Hom}}\newcommand{\Ext}{\operatorname{Ext}}\newcommand{\Cone}{\operatorname{Cone}}
\newcommand{\Ind}{\operatorname{Ind}}\newcommand{\Pic}{\operatorname{Pic}}
\newcommand{\rk}{\operatorname{rk}}\newcommand{\sig}{\operatorname{sign}}
\newcommand{\calA}{\mathcal A_k}\newcommand{\Free}{\mathbf Z_{\A^1}}
\newcommand{\la}{\langle}\newcommand{\ra}{\rangle}
\newcommand{\pf}[1]{\langle\!\langle #1\rangle\!\rangle}
\newcommand{\Esgn}{\mathcal E}
\newcommand{\Btil}{\widetilde B}
\newcommand{\AR}{\textup{(AR)}}\newcommand{\AS}{\textup{(AS)}}
\newcommand{\cl}{\operatorname{cl}}

\title{Algebraic vector bundles on cluster varieties}
\author{Haoyang Liu and Tianle Liu}
\date{}
\address{University of California, Santa Barbara}
\email{haoyangliu@ucsb.edu}
\address{University of Southern California}
\email{tianleli@usc.edu}
\begin{document}
\raggedbottom
\begin{abstract}
We study algebraic vector bundles on cluster varieties.
Over a perfect field of characteristic different from two, every
algebraic vector bundle is trivial on a cluster variety with principal
coefficients at an acyclic seed of irreducible finite Dynkin type.
The proof is based on integral contractions of cellular
$\mathbb A^1$-chain complexes and successive freezing.
We also construct smooth cluster varieties 
admitting indecomposable rank-two vector bundles, whose nontriviality
is detected by quadratic Euler classes. The cellular calculation also gives a criterion for stable freeness
and computes integral real cohomology and the images of cycle class maps.
\end{abstract}
\maketitle

\section{Introduction}\label{sec:intro}
For an affine variety, triviality of all algebraic vector bundles
is equivalent to freeness of all finitely generated projective modules
over its coordinate ring. A guiding problem is the generalized Serre
question: is every algebraic vector bundle on a topologically
contractible smooth affine complex variety trivial
\cite[Question~6 and Remark~7]{AOContractible}?

Every algebraic vector bundle on affine space is trivial by the
Quillen--Suslin theorem \cite{Quillen,Suslin}.
The generalized question has an affirmative answer in dimensions at
most two \cite{ADContractible}, for $\A^1$-contractible smooth affine
complex varieties \cite{AHWI}, and for topologically contractible
smooth rational affine complex threefolds \cite{LR}.

There are also examples of bundle triviality on varieties which are
not contractible. By a theorem of Murthy, every algebraic vector bundle
on $SL_2$ is trivial over any field, whereas $SL_2(\C)$ retracts onto
$SU(2)\simeq S^3$. We refer to \cite[Theorem~1.2]{NT} for this
result. The same reference records Swan's nontrivial rank-two vector
bundle on $SL_4(\C)$ in Theorem~1.3.

Cluster structures provide explicit coordinate systems on varieties
arising in Lie theory, including double Bruhat cells, open Richardson
varieties, and braid varieties \cite{BFZ,CGGLSS,GLS}.
Their exchange combinatorics also governs geometric invariants, as
the cohomology calculations of \cite{LS1,LS2} illustrate.
In this paper, we study the triviality problem for cluster varieties.
Throughout the paper all frozen variables are inverted,
and by a cluster variety we mean
$X=\operatorname{Spec}\mathcal A$, where $\mathcal A$ is a cluster
algebra of geometric type. Any seed gives a dense open torus
$\Gm^d$, and hence $k(X)=k(x_1,\ldots,x_d)$; in particular, every
such variety is rational
\cite[Section~1]{BFZ}.

We first consider cluster varieties with principal coefficients at
an acyclic seed of irreducible finite Dynkin type. Recall that in this
case the extended matrix is
$\widetilde B=\binom{B}{I_n}$, and the Cartan companion
$(2\delta_{ij}-|B_{ij}|)$ is an indecomposable finite-type Cartan
matrix. Our first main result is the following.

\begin{maintheorem}\label{thm:main-triviality}
Let $k$ be a perfect field of characteristic different from two.
Every algebraic vector bundle on a cluster variety with principal
coefficients at an acyclic seed of irreducible finite Dynkin type is
trivial.
\end{maintheorem}

As Swan observed, Quillen's method also proves that every finitely
generated projective module over $k[t_1^{\pm1},\ldots,t_d^{\pm1}]$
is free, for any field $k$ \cite[p.~111]{SwanLaurent}.
Consequently, every algebraic vector bundle on a cluster variety is
trivial on each seed torus. To prove global triviality, one must choose
such a trivialization so that it extends across the complement of the
seed torus.
In characteristic zero, Gross--Hacking--Keel describe
principal-coefficient cluster $\mathcal A$-spaces, up to codimension
two, as open subsets of blowups of toric varieties
\cite[Theorem~3.9(i)]{GHK}.

An application is given by reduced Coxeter double Bruhat cells.
Let $G$ be a simple simply connected complex algebraic group and let
$c$ be a Coxeter element. Yang--Zelevinsky identify the reduced double
Bruhat cell
$L^{c,c^{-1}}$ with a cluster variety with principal coefficients
at an acyclic seed of the same Dynkin type \cite[Theorem~1.2]{YZ}.
Thus every algebraic vector bundle on $L^{c,c^{-1}}$ is trivial.

The varieties in Theorem~\ref{thm:main-triviality} are not
topologically contractible over $\C$. Indeed, if $d=\dim X>0$, the
initial torus gives
$\gr^W_{2d}H^d(X(\C),\Q)=\Q(-d)$, and this class survives the
stratification spectral sequence (Corollary~\ref{cor:top-weight}).
By complex Betti realization, it also obstructs unstable and stable
$\A^1$-contractibility \cite{LevineBetti}. The theorem therefore
complements the contractible case of the generalized Serre question.

To prove the theorem, we construct integral cellular
$\A^1$-chain homotopies for paths and then apply successive freezing.
This gives $H^q_{\Nis}(X,M)=0$ for every strictly
$\A^1$-invariant sheaf $M$ and every $q>1$. Combining this vanishing
with Picard vanishing and the splitting theorem of Asok--Fasel
\cite[Theorem~1.1]{AFICM}, we obtain triviality in every rank.

Over the fields of Theorem~\ref{thm:main-triviality}, one can also
understand the $SL_2$ example by the same obstruction theory.
The first-column projection $SL_2\to\A^2\setminus\{0\}$ is a
Zariski locally trivial $\mathbb G_a$-torsor, hence an
$\A^1$-equivalence \cite[Example~3.7]{AFICM}.
Purity at the origin gives $H^q_{\Nis}(SL_2,M)=0$ for $q>1$
and every strictly $\A^1$-invariant sheaf $M$; moreover $\Pic(SL_2)=0$.
The splitting theorem therefore recovers bundle triviality in this
case as well.

The argument applies more generally to other coefficient patterns,
arbitrary additional type $A$ components, and support graphs with
chordal complement. We give these extensions in
Section~\ref{sec:triviality}.

\subsection{An indecomposable vector bundle}
Hartshorne's conjecture predicts that every rank-two vector bundle
on $\mathbb P^N_{\C}$ with $N\ge7$ splits as a sum of line bundles
\cite[Conjecture~6.3]{Hartshorne}.
It provides one motivation for studying indecomposable vector bundles of
small rank on smooth affine varieties of large dimension.
Asok--Fasel--Hopkins construct such bundles on split affine quadrics
and Jouanolou devices of projective spaces
\cite[Introduction and Theorem~2.2.15]{AFHplocal}.
We will obtain examples of this phenomenon among cluster varieties.

The main example has dimension seven. Let $X_7$ be the cluster
variety with mutable graph consisting of the two triangles $125$
and $345$, with two arrows on each edge oriented by increasing
vertex labels. Its two frozen coefficient rows are $(1,1,0,0,1)$
and $(0,0,1,1,0)$. The resulting acyclic seed defines a smooth affine
sevenfold with equations \eqref{eq:sevenfold-presentation}.

\begin{maintheorem}\label{thm:main-counterexample}
Let $k$ be a perfect field of characteristic different from two.
The smooth seven-dimensional cluster variety $X_7/k$ admits an
indecomposable rank-two vector bundle $E$ satisfying
$E\oplus\mathcal O_{X_7}\simeq\mathcal O_{X_7}^3$.
\end{maintheorem}

Theorem~\ref{thm:sevenfold-bundles} gives the full classification
of vector bundles on $X_7$. All positive-degree ordinary Chow
groups and the reduced $K_0$ group vanish. The Chow--Witt group
$\CW^2(X_7)$ is isomorphic to $W(k)^2$; the two summands come from
square cycles in induced independence complexes. Further examples,
including a six-dimensional product and variants associated with a
pentagon and a path, are discussed briefly in
Section~\ref{subsec:further-indecomposable-examples}.

\subsection{Cellular calculations and earlier work}
The proofs are based on a computation of the full cellular
Milnor--Witt differential, with its normal frames and quadratic
coefficients. For even exchange matrices, the calculation can be
expressed in terms of induced independence complexes. It determines
the Chow--Witt groups and, under an integral homology hypothesis,
the exact $\A^1$-cohomological dimension. We work in Morel--Sawant's
theory of cellular $\A^1$-homology \cite{MS}. Related calculations
for flag varieties, toric varieties, wonderful models, and
torus-enriched Bruhat complexes appear in
\cite{LL,LP,LPWonderful,LLTorus}.

Double braid varieties parametrize configurations of flags with
prescribed relative positions. Their coordinate rings carry really
full rank (RFR), sink-recurrent cluster structures, with frozen variables
inverted, for complex simple simply connected groups $G$ of every Dynkin type
\cite[Theorems~1.1 and~4.10, Corollary~6.7]{GLS}.
Here RFR means that the extended exchange matrix defines a split
injection of free abelian groups.
In particular, for opposite Borel subgroups $B_+,B_-$ of $G$ and
Weyl-group elements $v\le w$, the open Richardson variety
\[
 R^\circ_{v,w}=(B_-vB_+/B_+)\cap(B_+wB_+/B_+)
\]
is isomorphic to the braid variety
$X(\mathbf b(w)\mathbf b(v^{-1}w_0))$, where $w_0$ is the longest
element and $\mathbf b(a)$ is the positive braid lift of a reduced word for $a$
\cite[Theorem~3.14]{CGGLSS}. Its closure in $G/B_+$ is the
projective Richardson variety. Chow vanishing and stable freeness
hold throughout the RFR sink-recurrent class
\cite[Theorem~5.5 and Corollary~4.32]{Hlavinka}.
For double braid varieties of irreducible finite cluster type,
Corollary~\ref{cor:braid-bundles} gives triviality in every rank;
it also allows additional type $A$ components.

Under the anticlique rank and tameness condition \AR, we obtain a
component formula which characterizes Chow vanishing and stable
freeness; see Appendix~\ref{sec:component-chow}. To study the
individual vector bundles, we use the quadratic information in the
cellular differential.

In the appendices, we give the cellular differential and the path
contractions, together with the component and independence-complex
formulas. We also compute integral real cohomology and describe the
real cycle map as an inclusion of explicit signature lattices.
The same purity filtration gives mixed Tate motives for acyclic
seeds satisfying the anticlique saturation condition \AS.
For acyclic RFR seeds over $\Q$, combining this construction with
the known splitting of their mixed Hodge structures also gives
the rational motive decomposition of \cite[Theorem~4.29]{Hlavinka},
with multiplicities computed by the oriented cellular differential.
Appendix~\ref{sec:rational} explains this deduction and combines it
with real realization through the rational plus/minus decomposition,
as in \cite{LiuAbelian}.

The main text is organized as follows. In
Section~\ref{sec:geometry}, we describe the exchange geometry used
throughout the paper. In Section~\ref{sec:triviality}, we prove
Theorem~\ref{thm:main-triviality}. Sections~\ref{sec:bundles}
and~\ref{sec:sevenfold} contain the calculation and interpretation of
the obstruction in Theorem~\ref{thm:main-counterexample}.

\subsection{Conventions}
Schemes are separated and of finite type over the indicated field.
Except in the ordinary Chow calculation, the ground field is perfect
of characteristic different from two. We write $\CH$ for integral
Chow groups, $\CW$ for untwisted Chow--Witt groups, and
$\KM_r$, $\KMW_r$, $\Wsh$, and $\Ish^r$ for the unramified Milnor,
Milnor--Witt, Witt, and fundamental-ideal sheaves, with
$\Ish^r=\Wsh$ for $r\le0$. The contraction $M_{-1}$ is the kernel
of evaluation at $1$ from $M(-\times\Gm)$ to $M(-)$.
We use homological indexing for chain complexes. The symbol
$\SH(k)$ denotes the stable motivic infinity-category,
$\DM(k;\Q)$ the stable infinity-category of rational motives,
and $S^{a,b}=S_s^{a-b}\wedge\Gm^{\wedge b}$.
In these categories, cofibers and tensor products are understood in
the infinity-categorical sense; distinguished triangles lie in the
corresponding homotopy categories. For a motivic spectrum $E$, our grading is
$E^{a,b}(X)=\Hom_{\SH(k)}(\Sigma^\infty X_+,\Sigma^{a,b}E)$,
where $\Hom$ denotes morphisms in the homotopy category.

\subsection*{Acknowledgements}
The authors would like to thank Eugene Gorsky, Keyao Peng, Tonie Scroggin, and Lingyao Xie 
for helpful discussions. This material is based upon work supported by the National Science Foundation 
under Grant No. DMS-2424139, while the first author was in residence at the Simons Laufer 
Mathematical Sciences Institute in Berkeley, California, during the Fall 2026 semester. 

\section{Acyclic exchange geometry}\label{sec:geometry}
We begin with the exchange presentation and the geometry of its zero
strata. For an acyclic $d\times n$ extended matrix, the exchange equations are
\begin{equation}\label{eq:exchange-presentation}
 x_i x'_i=\prod_{a=1}^d x_a^{[\widetilde B_{ai}]_+}
          +\prod_{a=1}^d x_a^{[-\widetilde B_{ai}]_+}
 \quad(1\le i\le n),
\end{equation}
where $x_{n+1},\ldots,x_d$ are invertible and $[b]_+=\max(b,0)$.

\subsection{Independent exchange columns}\label{sec:independent-columns}
Let $\Btil$ be a $d\times n$ extended exchange matrix with acyclic
skew-symmetrizable principal part, with all frozen variables inverted.
We write $\Gamma$ for the undirected mutable graph, $\Ind(\Gamma)$ for its
independent sets (also called anticliques), and
$\alpha(\Gamma)=\max_{I\in\Ind(\Gamma)}|I|$.
For $I\in\Ind(\Gamma)$, let $B_I$ consist of the columns of $\Btil$
indexed by $I$, with the rows indexed by $I$ deleted. The two conditions
used below are the following:
\begin{description}
\item[\AR] \emph{Anticlique rank and tameness.}
Every $B_I$ has rank $|I|$, and its nonzero Smith invariant
factors are prime to $\operatorname{char}(k)$.
\item[\AS] \emph{Anticlique saturation.}
Every $B_I$ has rank $|I|$ and all its Smith invariant
factors equal one.
\end{description}
The tameness clause in \AR\ is vacuous in characteristic zero. Condition
\AS\ says that each independent set of exchange columns generates a
saturated sublattice; it implies \AR\ over every field. The empty set
satisfies both conditions.

Recall that an RFR matrix induces a split injection $\Z^n\hookrightarrow\Z^d$.
Each subcollection of its columns is therefore saturated. For an
independent set $I$, the deleted rows are zero on the retained columns,
so every acyclic RFR seed satisfies \AS. The converse fails both for
rational rank and for integral saturation; examples appear in
Appendix~\ref{sec:component-chow} and the sevenfold in
Section~\ref{sec:sevenfold}.

Let $A$ be the cluster algebra over $k$ with these inverted
coefficients, and put $X=\Spec A$.

For $I\in\Ind(\Gamma)$, set
\begin{equation}\label{cc:eq:torus-fiber}
 T_I=\{x\in\Gm^{[d]\setminus I}:\chi_i(x)=-1\text{ for }i\in I\},
 \qquad
 \chi_i=\prod_{a\notin I}x_a^{\Btil_{ai}}.
\end{equation}
Let $\mathcal O_I\subset X$ be the locally closed locus on which the original
mutable coordinates vanish exactly at $I$.

\begin{proposition}\label{cc:prop:geometry}
Under \AR, the variety $X$ is smooth and geometrically integral of
dimension $d$, only independent mutable zero sets occur, and
\[
 \mathcal O_I\cong\A^{|I|}\times T_I.
\]
If $e_1,\ldots,e_r$, where $r=|I|$, are the Smith factors of $B_I$, then
there is a noncanonical isomorphism
\begin{equation}\label{cc:eq:smith-stratum}
 T_I\cong E_I\times\Gm^{d-2|I|},\qquad
 E_I=\Spec\frac{k[z_1^{\pm1},\ldots,z_r^{\pm1}]}
 {(z_1^{e_1}-b_1,\ldots,z_r^{e_r}-b_r)},
 \quad b_i\in\{1,-1\}.
\end{equation}
Here $E_I$ is finite \'etale. The functions $x_i$, $i\in I$, are transverse
normal parameters along $\mathcal O_I$.
\end{proposition}

\begin{proof}
By \cite[Theorem~1.20 and Corollary~1.21]{BFZ}, the acyclic seed
gives the presentation \eqref{eq:exchange-presentation} of $A$, with
all frozen variables inverted. The anticlique stratification in this
presentation is the one used in \cite{LS2}. The standard monomial
basis shows that this presentation is preserved by field extension.
Over each extension the cluster algebra is a
domain with dense initial torus $\Gm^d$, so $X$ is geometrically
integral of dimension $d$.

We next describe the possible zero sets. Suppose the mutable zero set
contains an edge. The induced orientation is
acyclic and has a source in a nontrivial connected component. At that
source, one exchange monomial vanishes and the other does not. The
exchange equation contradicts the vanishing of the mutable coordinate.
Hence the zero set is independent. For independent $I$, the coordinates
outside $I$ are units, the exchange equations at $I$ give exactly
\eqref{cc:eq:torus-fiber}, and the mutated coordinates indexed by $I$ are
free. This proves the description of $\mathcal O_I$.

To describe the character fiber, we put $B_I$ into Smith form by integer
row and column operations. These identify
the character fiber with \eqref{cc:eq:smith-stratum}. The transformed
right-hand sides are products of powers of $-1$, so $b_i\in\{1,-1\}$.
Tameness makes each power equation \'etale. Thus $T_I$ is a disjoint union
of split tori over finite separable extensions of $k$; its finite \'etale
scheme of connected components is canonical, although the displayed
product decomposition is not.

It remains to prove smoothness and identify the normal parameters.
Near a point with mutable zero set $I$, we invert all
original mutable coordinates outside $I$. Their mutated coordinates are
then determined. Rescaling the remaining mutated coordinates by units
gives the equations
\[
 x_i y_i=1+\chi_i\quad(i\in I)
 \quad\text{over }S=\Gm^{[d]\setminus I}.
\]
The character map $S\to\Gm^I$ is smooth under \AR. The local model is its
base change along
\[
 \bigl(\A^2\setminus\{xy=1\}\bigr)^I\longrightarrow\Gm^I,
 \qquad (x_i,y_i)_i\longmapsto(x_i y_i-1)_i.
\]
This base change is smooth over the smooth source of the displayed map.
The equations $x_i=0$
are transverse coordinate equations there and remain so after smooth
base change. They give the asserted normal parameters and codimension
$|I|$.
\end{proof}

We fix an ordering of the coordinates $x_i$. This gives a framing of the normal
bundles in the anticlique filtration whose open terms $\Omega_r$ consist
of zero sets of cardinality at most $r$. Under \AS, each $E_I$ is
$\Spec k$, so every $T_I$ is geometrically connected and is a split torus.

\subsection{Stable classes and Chow groups}
The stratum description also determines the stable classes of vector
bundles under the saturation hypothesis.
\begin{proposition}\label{prop:kzero}
Under \AS, the rank map gives $K_0(X)=\Z$ and every algebraic vector
bundle on $X$ is stably free. Moreover $\CH^{>0}(X)=0$, and the same
assertions hold after every extension of the ground field.
\end{proposition}
\begin{proof}
We compute along the anticlique filtration. For $|I|=r>0$, the stratum $\mathcal O_I$ is the closed zero scheme
of the regular sequence $(x_i)_{i\in I}$ in the smooth open
$\Omega_r$. Its $G_0$ group is $\Z$, generated by its structure
sheaf, since $\mathcal O_I\cong\A^r\times\Gm^{d-2r}$.
The Koszul resolution gives
\[
 [\mathcal O_{\mathcal O_I}]
 =\sum_{a=0}^r(-1)^a\binom ra[\mathcal O_{\Omega_r}]=0.
\]
Localization for the union of strata with $|I|=r$ therefore gives
$G_0(\Omega_r)\cong G_0(\Omega_{r-1})$. Induction from the initial
torus proves $K_0(X)=G_0(X)=\Z$ by rank.

To deduce stable freeness, put $R=\Gamma(X,\mathcal O_X)$. If $[P]=r[R]$ for a projective
$R$-module, there is a projective $Q$
with $P\oplus Q\cong R^r\oplus Q$; adding a projective complement
of $Q$ then proves stable freeness.
The Chow vanishing follows from the residue calculation in
Proposition~\ref{prop:acyclic-chow}, or over an arbitrary field from
Theorem~\ref{cc:thm:chow}. The same geometry applies after field
extension.
\end{proof}
Appendix~\ref{sec:component-chow} gives the necessary and sufficient
connected-component criterion under \AR.

\section{Triviality by successive freezing}\label{sec:triviality}
\label{subsec:freezing-triviality}

Throughout this section $k$ is perfect of characteristic different
from two. All frozen variables are inverted. A path with single arrows
means a skew-symmetric path matrix whose nonzero entries are $1$ and
$-1$. A disjoint union of such paths may include isolated vertices and
the empty graph. During deletion we allow nonzero scalar coefficients
in the two terms of each exchange equation.

For a smooth variety $Y$, write $\operatorname{cd}_{\A^1}(Y)\le h$
if $H^q_{\Nis}(Y,M)=0$ for every strictly $\A^1$-invariant sheaf
$M$ and every $q>h$. We first recall the cellular calculation needed
in the argument. Under \AS,
the anticlique filtration gives a bounded projective complex in the
category $Ab_{\A^1}(k)$ of strictly invariant sheaves, with terms
\[
 C_0(X)=\mathbf H^{\A^1}_0(\Gm^d),\qquad
 C_p(X)=\bigoplus_{|I|=p}\KMW_p\otimes_{\A^1}
       \mathbf H^{\A^1}_0(T_I)\quad(p>0).
\]
The sum runs over independent sets, and $\otimes_{\A^1}$ is the
tensor product in that category.
By Theorem~\ref{thm:acyclic-model}, $\Hom(C_*(X),M)$ computes
$H^*_{\Nis}(X,M)$. Its differential is a sum over
$I\subset I\cup\{j\}$, obtained by taking the residue in $x_j$,
restricting the remaining torus characters, and including the change
of normal frame.
For a disjoint union of paths with single arrows,
Theorem~\ref{thm:integral-path-vanishing} contracts this complex
to degree zero. In particular such varieties have
$\operatorname{cd}_{\A^1}=0$, for all coefficient rows satisfying
\AS\ and with the scalar coefficients allowed above.

The following lemma describes the open and closed pieces arising from
freezing, together with the hypotheses inherited by those pieces.
We write $N_\Gamma[v]$ for the closed neighborhood of $v$, consisting
of $v$ and all its neighbors.

\begin{lemma}\label{lem:AS-framed-freezing}
Let $X$ have an acyclic exchange presentation satisfying \AS, and let
$v$ be a mutable vertex of its support graph $\Gamma$. Then
$U=D(x_v)$ is obtained by freezing $v$ and satisfies \AS. Moreover,
$D=V(x_v)$ is a smooth divisor with normal parameter $x_v$, and
\[
 D\cong\A^1\times Y,
\]
where $Y$ has an acyclic exchange presentation satisfying \AS, with
principal matrix obtained by restricting that of $X$ to
$V(\Gamma)\setminus N_\Gamma[v]$. Scalar coefficients can occur in
this presentation. Moreover,
$\operatorname{Pic}(X)=0$.
\end{lemma}
\begin{proof}
We first allow for the scalar coefficients introduced by freezing.
The local calculation of Proposition~\ref{cc:prop:geometry} applies
unchanged when the exchange-character equations are $\chi_i=c_i$ for
$c_i\in k^\times$: a primitive independent family of characters
extends to a torus-coordinate basis, and its fiber is a translated
split torus. This proves smoothness and the asserted normal framing.
The same strata have dimensions $d-|I|$, so the initial torus is
dense and the smooth $d$-dimensional variety is geometrically integral.
Freezing only removes a column, and therefore preserves \AS.

To identify the closed piece, note that every neighbor of $v$ is
invertible on $D$. We eliminate the corresponding exchange partners
and retain $x'_v$ as the free affine coordinate. Put
$J=V(\Gamma)\setminus N_\Gamma[v]$ and let $L$ be the character
lattice of the torus whose coordinates are the neighbors of $v$
and the original frozen variables. Put $M=L\oplus\Z^J$, and write
$b_i\in M$ for the exchange column indexed by $i\in J\cup\{v\}$,
with the $v$th row omitted. In particular, $b_v\in L$.
The equation at $v$ imposes the primitive character relation
$\chi_v=c$ in this torus, for a nonzero scalar $c$. Thus the remaining
frozen torus has character lattice $L/\Z b_v$.
For every independent $I\subseteq J$, the columns indexed by
$I\cup\{v\}$ form a primitive independent family in $M$. Consequently
\[
 (M/\Z b_v)/\langle\bar b_i:i\in I\rangle
 \cong M/\langle b_v,b_i:i\in I\rangle
\]
is free, and the images $\bar b_i$ are independent. This proves
\AS\ for $Y$. The principal matrix of $Y$ is the restriction of
that of $X$ to $J$, since $b_v$ has no coordinates in $J$. Choosing coordinates on the translated torus
accounts for the allowed scalar coefficients.

For Picard vanishing, recall that the complement of the initial torus
is the union of the
divisors $V(x_i)$. Each has a single irreducible dense singleton
stratum and multiplicity one in $\operatorname{div}(x_i)$.
Divisor localization from the initial torus, whose Picard group is
zero, therefore gives $\operatorname{Pic}(X)=0$.
\end{proof}

We can now apply purity with the displayed framing. For every strictly
$\A^1$-invariant sheaf $M$, this gives an exact sequence containing
\begin{equation}\label{eq:AS-framed-localization}
 H^{q-1}_{\mathrm{Nis}}(Y,M_{-1})\longrightarrow
 H^q_{\mathrm{Nis}}(X,M)\longrightarrow
 H^q_{\mathrm{Nis}}(U,M)\longrightarrow
 H^q_{\mathrm{Nis}}(Y,M_{-1}).
\end{equation}
It follows that
\[
 \operatorname{cd}_{\A^1}(X)\le
 \max\{\operatorname{cd}_{\A^1}(U),
       1+\operatorname{cd}_{\A^1}(Y)\}.
\]
If $Y$ has cohomological dimension zero, restriction is an
isomorphism in every degree at least two.

\begin{lemma}\label{lem:edgeless-strict-vanishing}
An acyclic presentation satisfying \AS\ with no mutable edges has
$\A^1$-cohomological dimension zero, also with nonzero scalar
coefficients.
\end{lemma}
\begin{proof}
Let $n$ and $m$ be the numbers of mutable and frozen variables,
respectively. In the absence of mutable edges, all mutable columns
are frozen characters and form a primitive
independent family. A change of frozen coordinates and scalar
rescaling identify the variety with
\[
 S^n\times\Gm^{m-n},\qquad S=\A^2\setminus\{xy=1\}.
\]
We verify that both factors preserve the required vanishing.
If $T$ has cohomological dimension zero, localization for
$\{0\}\times T\subset\A^1\times T$ proves the same for
$\Gm\times T$: the two groups bounding its positive-degree
cohomology are $H^q(T,M)$ and $H^q(T,M_{-1})$.
Localization for $\{xy=1\}\times T\subset\A^2\times T$ then
proves the assertion for $S\times T$. Both normal line bundles are
framed by their displayed equations. Induction starting with
$T=\operatorname{Spec}k$ proves the result.
\end{proof}

\begin{theorem}\label{thm:ordered-freezing-triviality}
Let $X$ have an acyclic seed satisfying \AS. Suppose its mutable
vertices have an ordering $v_1,\ldots,v_n$ for which the principal
matrix on
\[
 P_i=\{v_j:j>i,\ \{v_i,v_j\}\notin E(\Gamma)\}
\]
is a disjoint union of paths with single arrows for every $i$.
Then $\operatorname{cd}_{\A^1}(X)\le1$, and every algebraic
vector bundle on $X$ is free.
\end{theorem}
\begin{proof}
We argue by induction, freezing $v_1$ first. The open subset
$U=D(x_{v_1})$ inherits the remaining ordering.
By Lemma~\ref{lem:AS-framed-freezing}, we have
$V(x_{v_1})\cong\A^1\times Y$, where $Y$ satisfies \AS\ and has
principal matrix obtained by restricting that of $X$ to $P_1$.
Its cohomological dimension is zero by
Theorem~\ref{thm:integral-path-vanishing}. The induction hypothesis and
\eqref{eq:AS-framed-localization} therefore give the bound one.

It remains to pass from this vanishing to vector bundles.
Lemma~\ref{lem:AS-framed-freezing} gives Picard vanishing.
The splitting theorem \cite[Theorem~1.1]{AFICM} splits off a trivial
line bundle from every vector bundle of rank greater than one. Repeated
splitting reduces to a line bundle, which is trivial. This argument applies
over finite perfect fields as well as infinite fields.
\end{proof}

A graph is chordal if it has no induced cycle of length at least four.

\begin{corollary}\label{cor:cochordal-triviality}
Let $X$ have an acyclic seed satisfying \AS. If the complement of
its mutable support graph is chordal, every algebraic vector bundle
on $X$ is free. In particular this applies to
complete multipartite support graphs.
\end{corollary}
\begin{proof}
We use a perfect elimination ordering of the complement graph.
Its later neighbors form a clique, so the sets $P_i$ in
Theorem~\ref{thm:ordered-freezing-triviality} are independent.
In this case the proof uses only
Lemma~\ref{lem:edgeless-strict-vanishing}, without the path
calculation. The complement of a complete multipartite graph is
a disjoint union of complete graphs and is chordal.
\end{proof}

\begin{corollary}\label{cor:irreducible-finite-triviality}
Let $X$ have an acyclic seed of irreducible finite Dynkin type
whose extended matrix satisfies \AS. Every algebraic vector
bundle on $X$ is free. In particular the assertion holds for
every RFR cluster variety of irreducible finite cluster type,
and for every cluster variety with principal coefficients at
an acyclic seed of irreducible finite type.
\end{corollary}
\begin{proof}
We choose an ordering to which the preceding theorem applies.
Type $A$ is already a path with single arrows. In types $D$ and
$E$, we remove the unique branching vertex. In types $B$ and $C$,
we remove the terminal vertex incident to the multiple edge. In
type $F_4$, we remove either vertex incident to the multiple edge;
in type $G_2$, either vertex may be removed. In each case both the graph
with that vertex removed and the graph with its closed
neighborhood removed are disjoint unions of paths with single
arrows. Placing that vertex first and ordering the remaining
vertices arbitrarily gives the required ordering.

For the remaining assertions, recall that RFR is invariant under
exchange-matrix mutation, which is given
by left and right unimodular transformations. A cluster algebra
of irreducible finite type has an acyclic Dynkin seed, so the
second assertion follows. Principal coefficients give an
identity block and hence RFR. 
\end{proof}

This proves Theorem~\ref{thm:main-triviality}.
For type $A$, the stronger conclusion
$\operatorname{cd}_{\A^1}(X)=0$ follows directly from the path
contraction; the uniform bound for all irreducible types is one.

\begin{corollary}\label{cor:reduced-coxeter-bundles}
Let $G$ be a simple simply connected complex algebraic group and $c$
a Coxeter element in its Weyl group. Every algebraic vector bundle
on the reduced double Bruhat cell $L^{c,c^{-1}}$ is trivial.
\end{corollary}
\begin{proof}
The result follows by applying Theorem~\ref{thm:main-triviality} to
the principal-coefficient cluster structure of \cite[Theorem~1.2]{YZ}.
\end{proof}

\subsection{Further consequences of freezing}
The same ordering argument allows additional path components, even
when their frozen coefficients are shared.

\begin{corollary}\label{cor:finite-type-with-A-factors}
Let $X$ have an acyclic seed satisfying \AS\ whose Cartan companion
is of finite type. If at most one connected component is not of
type $A$, every algebraic vector bundle on $X$ is trivial.
The frozen coefficient rows may couple different components.
\end{corollary}
\begin{proof}
If all components have type $A$, the path theorem applies directly.
Otherwise we choose the vertex of the unique remaining component used
in the proof of Corollary~\ref{cor:irreducible-finite-triviality}.
Deleting it, or its closed neighborhood, leaves a disjoint union
of paths with single arrows, including all the original type $A$
components. With this vertex first, the others may be listed in any
order, and Theorem~\ref{thm:ordered-freezing-triviality} applies.
\end{proof}

\begin{corollary}\label{cor:braid-bundles}
Let $X$ be a complex double braid variety whose cluster algebra is
of finite type, with at most one non-$A$ connected component in its
Dynkin diagram.
Then every algebraic vector bundle on $X$ is trivial.
In particular, this holds for open Richardson varieties with such
a cluster structure.
\end{corollary}
\begin{proof}
The cluster structure is RFR by \cite[Corollary~6.7]{GLS}, and finite
type supplies an acyclic Dynkin seed \cite{FZII}. Since mutation
preserves RFR, this seed satisfies \AS, and
Corollary~\ref{cor:finite-type-with-A-factors} applies.
The identification with braid varieties gives the Richardson case
\cite[Theorem~3.14]{CGGLSS}.
\end{proof}

Partial freezing can also preserve a nonzero obstruction. Let $X$
have an acyclic seed satisfying \AS, and let $M$ be strictly
$\A^1$-invariant. If vertices $v_1,\ldots,v_s$ are frozen
successively and, at each step, the variety $Y$ of
Lemma~\ref{lem:AS-framed-freezing} has a disjoint union of paths with
single arrows as its mutable graph, then \eqref{eq:AS-framed-localization} gives
\[
 H^q_{\Nis}(X,M)\xrightarrow{\sim}
 H^q_{\Nis}(D(x_{v_1}\cdots x_{v_s}),M)\qquad(q\ge2).
\]
Indeed, each such $Y$ has cohomological dimension zero.
This isomorphism allows further bundle questions to be
reduced to smaller mutable graphs.

\section{Euler classes and vector bundles of small rank}\label{sec:bundles}
\subsection{Euler classes in cohomological dimension two}
Proposition~\ref{prop:kzero} gives stable freeness under \AS.
We now use cohomological bounds to classify vector bundles of small
rank and determine when adding one trivial line bundle makes them
free. Throughout this section $k$ is perfect of characteristic
different from two, and all cluster varieties have their frozen
variables inverted.
Recall that for a smooth affine variety $Y$, an orientation of a rank-$r$ vector
bundle $E$ is an isomorphism
$\theta:\det(E)\xrightarrow{\sim}\mathcal O_Y$.
We write $\operatorname{Vect}^{\mathrm{or}}_r(Y)$ for isomorphism classes
of such pairs $(E,\theta)$, with isomorphisms required to preserve
the orientation. The Euler class of an oriented rank-$r$ vector
bundle takes values in the untwisted Chow--Witt group:
\[
 e(E,\theta)\in\CW^r(Y)=H^r_{\Nis}(Y,\KMW_r).
\]

\begin{proposition}\label{prop:bundle-range}
Let $Y$ be a smooth affine variety with $K_0(Y)=\Z$ by rank and $\operatorname{cd}_{\A^1}(Y)\leq h$.
Every vector bundle on $Y$ of rank greater than $h$ is free.
If $h\leq2$, the Euler class gives a bijection
\[
\operatorname{Vect}^{\mathrm{or}}_2(Y)\xrightarrow{\sim}
H^2_{\Nis}(Y,\KMW_2)=\CW^2(Y).
\]
The set of unoriented rank-two vector bundles is the quotient by the action of $\Gamma(Y,\mathcal O)^\times$, where $u$ acts on the Euler class by $\la u\ra$.
Every rank-two vector bundle in this case becomes free after taking
the direct sum with a trivial line bundle.
In particular, if $h\leq1$, every vector bundle on $Y$ is free.
\end{proposition}
\begin{proof}
We first establish freeness above the cohomological dimension. The $K_0$
hypothesis implies stable freeness by the projective-complement argument
of Proposition~\ref{prop:kzero}; taking determinants also gives $\Pic(Y)=0$.
Affine representability identifies oriented rank-$r$ vector bundles with
$[Y,BSL_r]_{\A^1}$ for $r\geq2$, over finite fields as well as infinite
fields \cite[Theorem~2.5]{AHWIII}.
For $r\geq2$ the homotopy fiber of $BSL_r\to BSL_{r+1}$ is
$\A^{r+1}\setminus0$, which is $(r-1)$-connected. Its relative Postnikov
tower has uniqueness groups beginning in degree $r$ and existence
obstructions beginning in degree $r+1$, with strictly invariant
coefficient sheaves. The base $BSL_{r+1}$ is simply connected, so these
coefficients are untwisted. If $r>h$, all these groups vanish, and
stabilization is a bijection on maps from $Y$. Since the bundle is
stably free, repeated stabilization shows that it is already free in
rank $r$. The rank-one assertion follows from $\Pic(Y)=0$.

For the classification in rank two, recall that $BSL_2$ is simply
connected and has first nonzero homotopy sheaf $\pi_2^{\A^1}(BSL_2)=\KMW_2$.
Its second Postnikov map defines the obstruction Euler class.
By \cite[Remark~5.2.2 and Theorem~5.3.2]{AFEuler}, this agrees
with the Chow--Witt Euler class up to multiplication by a unit
in $\GW(k)$.
When $h\leq2$, the higher uniqueness groups and all existence
obstructions vanish, giving the displayed bijection. These Postnikov
identifications and the determinant action are the obstruction-theoretic
construction of Euler classes used in \cite[Sections~2 and~6]{AFThreefolds}.
The finite Krull dimension of $Y$ ensures convergence of the tower for
maps from $Y$; the smaller bound $h$ annihilates the indicated groups.
Changing a determinant trivialization by $u$ multiplies the Euler class
by $\la u\ra$. Finally, rank three is greater than $h$, so the direct sum
with a trivial line bundle is free.
\end{proof}

\subsection{Even exchange matrices}\label{subsec:even-coefficients}
We apply the preceding proposition to a class for which the cellular
differential gives explicit Chow--Witt groups.
Let $B$ be an acyclic skew-symmetric $n\times n$ matrix with even
entries, and let $\Gamma$ be its support graph. Choose a partition
$[n]=V_1\sqcup\cdots\sqcup V_q$ into nonempty cliques, and write
$c(i)=a$ for $i\in V_a$. Define the coefficient matrix by
$C_{ai}=1$ if $i\in V_a$, and $C_{ai}=0$ otherwise. Let $X_{B,c}$
be the cluster variety of $\Btil=\binom B C$, with all $q$ frozen
variables inverted. An independent set meets each clique at most
once, so its frozen rows contain an identity submatrix. Thus \AS\
holds and $\dim X_{B,c}=n+q$. Principal coefficients correspond to
the partition into singletons.

For $T\subseteq[n]$, write $\Ind(\Gamma[T])$ for the simplicial
complex of independent sets of the induced graph. We use augmented
reduced homology, with
$\widetilde H_{-1}(\Ind(\Gamma[\varnothing]);\Z)=\Z$, and put
\[
w_c(T)=\prod_{a\notin c(T)}(1+2^{|V_a|}),\qquad
\lambda(\Gamma)=\max\{p\ge0:
\widetilde H_{p-1}(\Ind(\Gamma[T]);\Z)\ne0
\text{ for some }T\}.
\]
The torsion-freeness condition is
\begin{equation}\label{eq:graph-TF}
\widetilde H_j(\Ind(\Gamma[T]);\Z)\text{ is free abelian for every }T\subseteq[n],\ j\in\Z.
\tag{TF}
\end{equation}

\begin{corollary}\label{cor:even-graph-summary}
Over a perfect field $k$ of characteristic different from two, for
every perfect extension $K/k$ and every $p>0$ there is an isomorphism
\begin{equation}\label{eq:even-summary-cw}
\CW^p((X_{B,c})_K)\simeq
\bigoplus_{T\subseteq[n]}
\widetilde H^{p-1}(\Ind(\Gamma[T]);\Witt(K))^{\oplus w_c(T)}.
\end{equation}
If \eqref{eq:graph-TF} holds, then
\[
\operatorname{cd}_{\A^1}((X_{B,c})_K)=\lambda(\Gamma).
\]
For principal coefficients, the formula becomes
\begin{equation}\label{eq:principal-even-cw}
\CW^p((X_{B,c})_K)\simeq
\bigoplus_{T\subseteq[n]}
\widetilde H^{p-1}(\Ind(\Gamma[T]);\Witt(K))^{\oplus3^{n-|T|}}.
\end{equation}
\end{corollary}
\begin{proof}
Theorem~\ref{thm:even-graph-CW} gives the Chow--Witt formula.
Under the homology hypothesis, Theorem~\ref{thm:even-cd} gives the
cohomological dimension. For the partition into singletons, each
vertex outside $T$ contributes $1+2=3$ to $w_c(T)$.
\end{proof}

Under \eqref{eq:graph-TF}, put $\lambda=\lambda(\Gamma)$.
Theorem~\ref{thm:graph-bundles} gives freeness above rank
$\lambda$ and in odd rank $\lambda$. In positive even rank
$\lambda$, the Euler class gives a bijection from oriented bundle classes to
$\CW^\lambda(X_{B,c})$;
forgetting orientation gives the orbits of the full global unit
group, computed in Proposition~\ref{prop:even-unit-action}.
In characteristic zero, Corollary~\ref{cor:graph-later-ranks}
extends odd-rank freeness and even-rank Euler classification to
$r\ge4$ with $\lambda\le r+2$, and gives freeness in rank three
when $\lambda\le4$.
Corollary~\ref{cor:graph-stabilization-realization} proves
$E\oplus\mathcal O\cong\mathcal O^{r+1}$ in wider ranges and,
for the stated even ranks, gives surjectivity of the Euler-class
map: every class in $\CW^r(X_{B,c})$ is realized by an oriented vector
bundle.

\section{A seven-dimensional example}\label{sec:sevenfold}
We now construct the example in the introduction and determine its
vector bundles using the Euler class.
Let $\Gamma$ consist of the two triangles $125$ and $345$, sharing
the vertex $5$. Take $B_{ij}=2$ for $i<j$ on the edges
$12,34,15,25,35,45$, with all other upper-triangular entries zero,
and use the coefficient cliques
\[
V_1=\{1,2,5\},\qquad V_2=\{3,4\},\qquad
C=\begin{pmatrix}1&1&0&0&1\\0&0&1&1&0\end{pmatrix}.
\]
With $t_1,t_2$ inverted, the resulting cluster variety $X_7=X_{B,c}$
has exchange presentation
\begin{equation}\label{eq:sevenfold-presentation}
\begin{aligned}
x_1y_1&=t_1+x_2^2x_5^2, &
x_2y_2&=t_1x_1^2+x_5^2,\\
x_3y_3&=t_2+x_4^2x_5^2, &
x_4y_4&=t_2x_3^2+x_5^2,\\
x_5y_5&=t_1x_1^2x_2^2x_3^2x_4^2+1.
\end{aligned}
\end{equation}
Condition \AS\ holds by the clique construction, so
Proposition~\ref{cc:prop:geometry} shows that $X_7$ is smooth and
geometrically integral of dimension seven. The reduction of $\Btil$
modulo two has rank $2<5$ over $\F_2$, so $\Btil$ is not RFR.

\begin{theorem}\label{thm:sevenfold-bundles}
Let $k$ be a perfect field of characteristic different from two.
The cluster variety $X_7/k$ satisfies
\[
\operatorname{cd}_{\A^1}(X_7)=2,\qquad
K_0(X_7)=\Z,\qquad \CH^{>0}(X_7)=0.
\]
The Euler class and forgetting orientation give bijections
\begin{equation}\label{eq:sevenfold-classification}
\operatorname{Vect}^{\mathrm{or}}_2(X_7)\simeq\CW^2(X_7)\simeq\Witt(k)^2,
\qquad
\operatorname{Vect}_2(X_7)\simeq\Witt(k)^2/k^\times,
\end{equation}
where $u\in k^\times$ acts simultaneously on both coordinates by
multiplication by $\la u\ra$. The zero orbit corresponds to the trivial vector bundle,
and every nonzero orbit represents an indecomposable rank-two
vector bundle. Every rank-two vector bundle $E$ satisfies
\[
E\oplus\mathcal O_{X_7}\simeq\mathcal O_{X_7}^3,
\]
and every vector bundle of positive rank other than two is free. These statements
hold after every perfect field extension of $k$.
Over an algebraically closed field, rank-two vector bundles have exactly
four isomorphism classes: the trivial vector bundle and three
indecomposable vector bundles.
\end{theorem}
\begin{proof}
We begin with the Chow--Witt calculation. The independence complex
of $\Gamma$ is the square with cyclic
vertex order $1,3,2,4$, disjoint from the vertex $5$. The only
induced independence complexes with nonzero first reduced
cohomology are those on
\[
T_0=\{1,2,3,4\},\qquad T_1=[5].
\]
Each contributes one copy of $\Witt(k)$. Both sets meet both
coefficient cliques, so $w_c(T_0)=w_c(T_1)=1$.
Corollary~\ref{cor:even-graph-summary} gives
$\CW^2(X_7)\simeq\Witt(k)^2$.

To apply the bundle classification, observe that $\alpha(\Gamma)=2$,
so the cellular complex displayed in
Section~\ref{sec:triviality} has length two. Hence
$\operatorname{cd}_{\A^1}(X_7)\le2$, and the nonzero group $\CW^2(X_7)$
gives equality. Proposition~\ref{prop:kzero} gives $K_0(X_7)=\Z$
and positive Chow vanishing. Proposition~\ref{prop:bundle-range}
then supplies the oriented classification, freeness in the other
ranks, and freeness after taking the direct sum with a trivial line
bundle in rank two.

We next compute the unit action in order to forget the orientation.
On each of the
two summands of $\CW^2(X_7)\simeq\Witt(k)^2$, the difference between $\la-t_a\ra$ and the
identity is the restriction to
$\Ind(\Gamma[T_i\setminus V_a])$, by
Proposition~\ref{prop:even-unit-action}. Removing either clique
destroys the square, so this restriction is zero in degree one.
Thus $\la-t_a\ra$ acts identically and $\la t_a\ra$ acts by
simultaneous multiplication by $-1$ on $\Witt(k)^2$.
The same proposition gives
\[
\mathcal O(X_7)^\times
=k^\times\times t_1^{\Z}\times t_2^{\Z}.
\]
The constant $-1$ already induces this action, so the orientation
orbits are precisely the scalar orbits in
\eqref{eq:sevenfold-classification}.

It remains to identify the indecomposable bundles. Since $\Pic(X_7)=0$,
a decomposable rank-two vector bundle is free.
Every nonzero Euler class remains nonzero under the unit action
and therefore represents an indecomposable vector bundle. These classes
exist over every stated field, since $\la1\ra\ne0$ in $\Witt(k)$.
The calculation commutes with perfect field extension. Over an
algebraically closed field, $\Witt(k)=\Z/2$ and all scalars are
squares, giving the four asserted classes.
\end{proof}

This proves Theorem~\ref{thm:main-counterexample}.
The nonfree vector bundles have vanishing ordinary Chern classes; their
nontriviality is detected by the quadratic Euler class.

\subsection{Further examples}
\label{subsec:further-indecomposable-examples}
The preceding construction has several variants.
The disjoint union of two edges, each with weight two and its own
coefficient clique, gives the product of two cluster threefolds.
This six-dimensional variety has $\CW^2\simeq\Witt(k)$ and
$\operatorname{Vect}_2\simeq\Witt(k)/k^\times$.

The square of the fourfold in Lam--Speyer
\cite[Section~6.3, Table~2, left]{LS1} gives an eight-dimensional
example. After reversing the three mutable indices, the fourfold has the
presentation
\begin{equation}\label{eq:LS-fourfold}
 X=\{x_1x'_1=t+x_2x_3,\quad x_2x'_2=x_1+x_3,\quad
       x_3x'_3=x_1x_2+1\},\qquad t\ne0.
\end{equation}
For $Z=X\times X$, Lemma~\ref{lem:LS-square} and
Proposition~\ref{prop:bundle-range} give
\[
 \CW^2(Z)\simeq\Witt(k),\qquad
 \operatorname{Vect}_2(Z)\simeq\Witt(k)/k^\times,
\]
where the action is $q\mapsto\la\lambda\ra q$.
Over an algebraically closed field $Z$ therefore has exactly one
nonfree rank-two vector bundle.

For another example, consider the even pentagon $C_5$. The coefficient cliques
$\{1,2\},\{3,4\},\{5\}$ give an eightfold with the same
classification. Principal coefficients give a tenfold, again with
Chow--Witt group $\CW^2\simeq\Witt(k)$. In both cases, the full pentagon is
the only induced subgraph whose independence complex has nonzero first
reduced cohomology. Its contribution has multiplicity one; removing a
coefficient clique kills that cohomology.

The principal-coefficient even path $P_5$ is another tenfold, with
$\CW^2\simeq\Witt(k)^4$. Its contributing vertex sets are
$\{1,2,4,5\}$ and $[5]$, with multiplicities three and one in
\eqref{eq:principal-even-cw}. The full independence complex collapses
onto a square, and all induced complexes have free integral homology
concentrated in degrees at most one. Corollary~\ref{cor:even-graph-summary}
therefore gives cohomological dimension two.
In each of these examples, the orbits of nonzero classes in $\CW^2$ under the
unit action give indecomposable rank-two vector bundles. Their direct
sums with a trivial line bundle are free, and every vector bundle of
positive rank other than two is free.

These examples, including $X_7$, satisfy \AS\ but their displayed
seeds fail the ordering condition of
Theorem~\ref{thm:ordered-freezing-triviality}.
For $X_7$, the six-dimensional product and the even path $P_5$,
an induced subgraph consists of two disjoint multiple edges;
its earliest vertex in any ordering has the opposite edge in its $P_i$.
For the even pentagon, $P_1$ contains the multiple edge between the
first vertex's two nonneighbors; for the square of the fourfold
\eqref{eq:LS-fourfold}, it contains the triangle from the other factor.

More generally, tensoring the cellular splitting of
Lemma~\ref{lem:LS-square} computes the Chow--Witt groups of every
power $X^m$ of the fourfold \eqref{eq:LS-fourfold} and expresses the cohomology
groups governing Postnikov obstructions and lifting ambiguities for
oriented vector bundles in terms of contractions of strictly
$\A^1$-invariant coefficient sheaves evaluated at $k$.

\appendix
\section{Cellular complexes and their differentials}\label{sec:strict}\label{sec:acyclic}
In this appendix we construct the cellular complexes used in the main text and compute their differentials. Throughout, $k$ is perfect and $\operatorname{char}k\neq2$. We write $\calA=Ab_{\A^1}(k)$ for the abelian category of strictly $\A^1$-invariant Nisnevich sheaves, denote its tensor product by $\otimes_{\A^1}$, and put $\Free[T]=\mathbf H^{\A^1}_0(T)$.

We use homological indexing for chain complexes: a sheaf $A[r]$ is concentrated in degree $r$. Cohomology of $\Hom(P_*,M)$ has the usual cohomological indexing.

\subsection{Strictness and mapping cones}
Let $J:D(\calA)\to D_{\A^1}(k)$ be the functor obtained by viewing a complex of strictly invariant sheaves as a complex of Nisnevich sheaves and then applying $\A^1$-localization. For $D\in D_{\A^1}(k)$, a strict resolution is an object $P\in D(\calA)$ with a map $D\to J(P)$ inducing
\begin{equation}\label{eq:strict-up}
\Hom_{D(\calA)}(P,E)\xrightarrow{\sim}
\Hom_{D_{\A^1}(k)}(D,J(E))
\end{equation}
for all $E\in D(\calA)$. We use the bounded-below form of the definition and the sheaf-and-shift test of \cite[Definition 2.33 and Lemma 2.35]{MS}.

\begin{lemma}\label{lem:cone}
Suppose $A\to B\to C\to A[1]$ is a distinguished triangle of bounded-below $\A^1$-complexes. If $A$ and $C$ have bounded projective strict resolutions $P_A,P_C$, then $B$ has a bounded projective strict resolution
\[
P_B=\Cone(P_C[-1]\longrightarrow P_A).
\]
The derived arrow is the one induced by $C[-1]\to A$ and the strict universal properties.
\end{lemma}
\begin{proof}
Composing $C[-1]\to A\to J(P_A)$ and applying \eqref{eq:strict-up} to $C[-1]$, we obtain a unique derived arrow $P_C[-1]\to P_A$. Bounded projectivity allows us to represent this arrow by a chain map. The axioms for morphisms of distinguished triangles then give the comparison $B\to J(P_B)$. To verify strictness, we test this comparison on a sheaf $M$ and every shift. Strictness identifies the outer terms in the long exact Hom sequences of the two triangles, and the five lemma identifies the middle terms. The sheaf-and-shift test therefore proves strictness. Finally, the terms of the cone are finite sums of the given projective terms.
\end{proof}
This is the representability argument of \cite[Proposition 2.37]{MS}, applied to a general purity triangle.

\subsection{Torus and Thom terms}
\begin{lemma}\label{lem:torus-strict}
For $T=\Gm^t$, $H^q_{\Nis}(T,M)=0$ for every $M\in\calA$ and $q>0$. The sheaf $\Free[T]$ is projective. For $r\geq1$,
\[
\bigl(\KMW_r\otimes_{\A^1}\Free[T]\bigr)[r]
\]
is a projective strict resolution of $\widetilde C^{\A^1}_*(S^{2r,r}\wedge T_+)$. For $r=0$ the resolution is $\Free[T]$ in degree zero.
\end{lemma}
\begin{proof}
A split torus is cohomologically trivial for strictly invariant coefficients, by the contraction and localization calculation; see \cite[Remarks 2.10 and 2.26(1), Example 2.36]{MS}. More explicitly, evaluation at $1$ and successive contractions split $M(T)$ into evaluations of contractions of $M$ at $k$. Both contraction and evaluation at a field are exact on $\calA$. Therefore $\Free[T]$, and the object representing $M\mapsto M_{-r}(T)$, are projective. The latter object is $\KMW_r\otimes_{\A^1}\Free[T]$ for $r\geq1$. The Thom-space computation in \cite[Example 2.36]{MS} gives the shift $r$ and its strict comparison.
\end{proof}
A choice of torus coordinates gives the decompositions
\begin{align}
\Free[\Gm^t]&\simeq\Z\oplus\bigoplus_{a=1}^t(\KMW_a)^{\oplus\binom ta},\label{eq:free-torus}\\
\KMW_p\otimes_{\A^1}\Free[\Gm^t]&\simeq
\bigoplus_{a=0}^t(\KMW_{p+a})^{\oplus\binom ta}\quad(p\geq1).\label{eq:thom-torus}
\end{align}
The degree-zero summand of \eqref{eq:free-torus} is $\Z$, not $\KMW_0$.

We use Lemma~\ref{lem:cone} for the purity triangle of a smooth
principal divisor $V(f)\cong\A^1\times Y$ with open complement $U$.
Its normal frame is $f$, and the chain model is the cone of the
attaching map from the Thom resolution of $Y$, shifted down once,
to the resolution of $U$. This description retains the actual
attaching map, which will be contracted in Appendix~\ref{app:path-vanishing}.

\subsection{The anticlique filtration}
Recall that $\Omega_p=\bigcup_{|I|\le p}\mathcal O_I$ and that
\begin{equation}\label{eq:strata}
\mathcal O_I\cong\A^{|I|}\times T_I.
\end{equation}
Proposition~\ref{cc:prop:geometry} identifies the ordered functions
$(x_i)_{i\in I}$ as normal frames whenever \AR\ holds.

\begin{proposition}\label{prop:AR-coefficients}
Under \AR, cohomology with a strictly $\A^1$-invariant sheaf $M$ is
computed by the purity cochain complex with terms
\[
C^p(X,M)=\bigoplus_{|I|=p}M_{-p}(T_I).
\]
In particular, $\operatorname{cd}_{\A^1}(X)\le\alpha(\Gamma)$.
\end{proposition}
\begin{proof}
Every $T_I$ is a finite disjoint union of split tori over finite
separable extensions of $k$. Evaluation on these fields is
Nisnevich-acyclic, and contraction together with split localization
proves $H^q_{\Nis}(T_I,N)=0$ for $q>0$ and every strictly invariant $N$.
Since affine factors do not change cohomology, the purity spectral sequence
for $\Omega_\bullet$ has only its degree-zero stratum row.
The framing of the codimension-$p$ normal bundle identifies its coefficients
with $M_{-p}$, so the resulting row complex proves the assertion.
\end{proof}

We now assume \AS. In this case $T_I$ is a split torus of dimension $d-2|I|$,
and the same construction is represented by a bounded projective
cellular complex.

\begin{theorem}\label{thm:acyclic-model}
Under \AS, the filtration $\Omega_\bullet$ is an oriented cellular structure in the sense of Morel--Sawant. Its chain groups are
\[
C_0^\Sigma(X)=\Free[\Gm^d],\qquad
C_p^\Sigma(X)=\bigoplus_{\substack{I\in\Ind(\Gamma)\\ |I|=p}}
\KMW_p\otimes_{\A^1}\Free[\Gm^{d-2p}]\quad(p>0).
\]
Its differential is the purity boundary, computed in Theorem~\ref{thm:full-differential} below. This complex is a strict resolution, and
$H^q_{\Nis}(X,M)=0$ for $q>\alpha(\Gamma)$ and all strictly invariant $M$.
Moreover, $\Sigma^\infty X_+$ is finite cellular.
\end{theorem}
\begin{proof}
Proposition~\ref{cc:prop:geometry} and the split-torus acyclicity establish the hypotheses of \cite[Propositions 2.27 and 2.37]{MS}. Their construction gives the displayed terms and the universal coefficient assertion, and the length of this complex is the maximal independent-set size. The same framed purity cofibers are suspensions of split tori and hence finite sums of motivic spheres. Induction along the filtration therefore proves finite cellularity.
\end{proof}

\subsection{The complete cellular differential}
To compute the differential, we first fix the residue conventions. Put $g=\la-1\ra$ and $\epsilon=-g$. We use left residues, normalized by
$\partial_u([u]\alpha)=\overline\alpha$ when $\alpha$ is unramified at $u=0$.
For successive regular immersions we order the new normal direction before
the previous normal frame.  Thus, if $I=\{i_1<\cdots<i_p\}$ and
$J=I\cup\{j\}$, comparison of $(j,i_1,\ldots,i_p)$ with the increasing
order on $J$ acts by $g^{\nu(I,j)}$, where
\[
\nu(I,j)=\#\{i\in I:i<j\}.
\]
This fixes the purity and contraction conventions throughout this
subsection.  The Milnor--Witt residue and commutativity conventions are
those of \cite[Lemma 1.6 and Theorem 1.7]{Fasel}.

For $I\in\Ind(\Gamma)$ write
\[
T_I=\left\{(x_a)_{a\notin I}\in\Gm^{[d]\setminus I}:
 \chi_i(x)=-1\ (i\in I)\right\},\qquad
\chi_i(x)=\prod_{a\notin I}x_a^{\Btil_{ai}}.
\]
The projection $\mathcal O_I\simeq\A^{|I|}\times T_I\to T_I$
forgets the free mutated coordinates.  If $J=I\cup\{j\}$ is independent,
put
\[
S_{I,j}=\left\{(x_a)_{a\notin J}\in\Gm^{[d]\setminus J}:
 \prod_{a\notin J}x_a^{\Btil_{ai}}=-1\ (i\in I)\right\}.
\]
Independence gives $\Btil_{ji}=0$ for $i\in I$, so that
$T_I=\Gm\times S_{I,j}$, with first coordinate $x_j$, and the additional equation
$\chi_j=-1$ defines a closed immersion
$\iota_{I,j}:T_J\hookrightarrow S_{I,j}$.
For any strictly $\A^1$-invariant sheaf $N$, write
\[
\rho_{I,j}:N(\Gm\times S_{I,j})\longrightarrow N_{-1}(S_{I,j})
\]
for the contraction projection in the splitting at $x_j=1$, with the
left-residue normalization just specified.

\begin{theorem}\label{thm:full-differential}
Let $k$ be perfect of characteristic different from two, and let $X$ be
the acyclic cluster variety of a seed satisfying \AS. For every strictly
$\A^1$-invariant sheaf $M$, the cochain complex
$\Hom_{\calA}(C_*^\Sigma(X),M)$ has terms
\[
C^p(X,M)=\bigoplus_{|I|=p}M_{-p}(T_I)
\]
and differential
\begin{equation}\label{eq:full-cellular-differential}
(d^pc)_J=\sum_{j\in J}g^{\nu(J\setminus\{j\},j)}
 \iota_{J\setminus\{j\},j}^{*}
 \rho_{J\setminus\{j\},j}(c_{J\setminus\{j\}}).
\end{equation}
The action of $g$ is its canonical action on the contracted sheaf
$M_{-p-1}$.  Formula~\eqref{eq:full-cellular-differential}, natural in
$M$, determines every differential of $C_*^\Sigma(X)$.
\end{theorem}
\begin{proof}
Fix $J=I\cup\{j\}$. We work in the open set where $x_a\ne0$ for
$a\notin J$ and impose $x_i=0$ for $i\in I$. We denote the resulting
smooth scheme by $W_{I,j}$.  Its open part $x_j\ne0$ is
$\mathcal O_I$, and its closed divisor $x_j=0$ is $\mathcal O_J$.
The usual acyclic presentation gives a morphism
\[
s:\A^1_u\times T_J\longrightarrow W_{I,j}
\]
by setting $x_j=u$, $x'_j=0$, $x'_i=0$ for $i\in I$, and taking all
remaining original coordinates from $T_J$.  The exchange relations for
the vertices in $J$ hold because $J$ is independent and $\chi_i=-1$ on
$T_J$; for the other vertices they determine the mutated coordinates,
since their original coordinates are invertible.  Thus $s$ is a regular
morphism, including at $u=0$.

The inverse image of $\mathcal O_J$ is $\{0\}\times T_J$, and the map
on normal line bundles is the identity in the parameters $u$ and $x_j$.
On closed strata, $s$ is the zero section
$T_J\to\A^{p+1}\times T_J$, an $\A^1$-equivalence.  On open strata,
after forgetting the affine factors, it is the map
\begin{equation}\label{eq:torus-face-map}
h_{I,j}:\Gm\times T_J\longrightarrow T_I,
\qquad x_j=u,\quad x_a=t_a\quad(a\notin J).
\end{equation}
Naturality of purity for this transverse morphism of pairs identifies
the adjacent boundary with contraction in $u$, followed by restriction
to $T_J$. The resulting map is $\iota_{I,j}^*\rho_{I,j}$. Adding the previous
normal directions gives the frame $(j,i_1,\ldots,i_p)$, whose comparison
with the increasing frame on $J$ contributes exactly
$g^{\nu(I,j)}$. Since no exchange monomial changes the normal parameter
$x_j$, there is no further unit factor. The closure relation for
the strata makes every other block zero, and the formula follows.

The quadratic signs also allow us to verify directly
that $d^2=0$.  If $j<\ell$ are both added to $I$, the independent-set
condition makes the two restrictions commute.  Interchanging two
contractions acts by $\epsilon$, since interchanging the reduced
factors of $\Gm\wedge\Gm$ acts by $\epsilon$ on $\KMW_2$.
The exponents of $g$ along the two paths differ by one.  Their sum is
therefore a common factor times $g+\epsilon=0$.  In the basic square,
the two iterated residues of $[x_j][x_\ell]$ are $1$ and $\epsilon$,
and their orientation factors are $g$ and $1$, respectively.
Finally, the construction and formula are natural in $M$; Yoneda's
lemma in $\calA$ determines the underlying sheaf morphisms uniquely.
\end{proof}

\begin{corollary}\label{cor:full-cellular-matrix}
After choosing torus coordinates on each $T_I$, the entire cellular
complex is a finite matrix of morphisms between the summands in
\eqref{eq:free-torus}--\eqref{eq:thom-torus}.  Its entries are computable
by integer linear algebra and Milnor--Witt symbol identities from
$\Btil$.  They can be chosen in the graded subring of
$K_*^{\mathrm{MW}}(k)$ generated by $\eta$ and $[-1]$.
\end{corollary}
\begin{proof}
Saturation permits an integral splitting of every character lattice.
An integral splitting sends $(-1)_{i\in I}$ to a base point of $T_I$
whose coordinates are all $\pm1$.  Smith normal form consequently gives
torus coordinates in which each map \eqref{eq:torus-face-map} is a
signed Laurent monomial map.

We describe the entries explicitly, keeping track of multiplication
orders. Choose ordered coordinates $z_1,\ldots,z_{d-2p}$ on $T_I$ and
$w_1,\ldots,w_{d-2p-2}$ on $T_J$. Let $A$ and $B$ index subsets of
these coordinates, and put $z_A=\prod_{a\in A}^{\mathrm{increasing}}[z_a]$.
The entries are defined by the expansion
\begin{equation}\label{eq:full-matrix-entry}
 g^{\nu(I,j)}\rho_u\bigl(h_{I,j}^*z_A\bigr)
   =\sum_B\left(\prod_{b\in B}^{\mathrm{increasing}}[w_b]\right)
      m_{I,A;J,B}.
\end{equation}
This expansion uses the torus contraction decomposition, with constant coefficients
on the right. Each coefficient satisfies
\[
m_{I,A;J,B}\in K^{\mathrm{MW}}_{|A|-|B|-1}(k)
\]
and is the matrix entry from the $B$-summand
$\KMW_{p+1+|B|}$ of $C_{p+1,J}$ to the $A$-summand
$\KMW_{p+|A|}$ of $C_{p,I}$.  Indeed, one obtains this entry by applying
\eqref{eq:full-cellular-differential} with $M=\KMW_{p+|A|}$ to the
identity on the indicated target summand.  For $a,b\geq1$, the universal
property gives
\[
\Hom_{\calA}(\KMW_a,\KMW_b)
 \simeq(\KMW_b)_{-a}(k)=K^{\mathrm{MW}}_{b-a}(k);
\]
the element on the right acts by right multiplication.  All source
indices here are positive.  The exceptional target $\Z\subset C_0$
has zero incoming block, since $\Z_{-a}=0$ for $a>0$; it is not treated
as $\KMW_0$.

To compute \eqref{eq:full-matrix-entry}, we substitute the signed Laurent
monomials and use the identities $[ab]=[a]+[b]+\eta[a][b]$ and
$[a^r]=r_\epsilon[a]$, where
\[
r_\epsilon=\sum_{v=0}^{r-1}g^v\quad(r\geq0),\qquad
r_\epsilon=-g\sum_{v=0}^{-r-1}g^v\quad(r<0),
\]
and reduce repeated symbols using
$[a]^2=[a][-1]$. We then move the $[u]$ factor to the left by
$\epsilon$-commutativity and take its residue. This requires only
finitely many operations, in which the only constants are $1$ and $-1$.
Consequently, the procedure computes every entry over the claimed coefficient ring.

There is also a useful description of each block in adapted coordinates.
The columns indexed by $J$ are saturated in $\Z^{[d]\setminus J}$.
After quotienting by the columns indexed by $I$, the character $\chi_j$
is therefore primitive.  Together with the independent coordinate $x_j$,
it extends to a character basis of $T_I$.  The resulting coordinates have
the form $(z,w,t)$, with $z=x_j$, $w=\chi_j$, and $t$ coordinates on
$T_J=\{w=-1\}\subset S_{I,j}$.  In these coordinates the face map is
$(u,t)\mapsto(u,-1,t)$.  Its reduced two-coordinate block is
\begin{equation}\label{eq:basic-MW-block}
\begin{split}
\Delta:\KMW_1&\longrightarrow
\Z\oplus\KMW_1\oplus\KMW_1\oplus\KMW_2,\\
v&\longmapsto(0,v,0,v[-1]),
\end{split}
\end{equation}
where the target coordinates are ordered as $1,z,w,z\wedge w$.
Tensoring with the remaining torus and Thom factors, with the specified
frame comparison, gives the corresponding full block.
\end{proof}

The torus-coordinate calculation is related to Liu--Peng's cubical formulas
\cite[Example 2.20 and Proposition 2.23]{LP}. Liu--Liu compute boundary
maps in all degrees for basic affine spaces by retaining torus translations and
Milnor--Witt frame degrees \cite[Theorem 5.12]{LLTorus}. The transverse
slice in Theorem~\ref{thm:full-differential} supplies the geometric boundary
map for cluster strata, including the identity and $[-1]$ components
of \eqref{eq:basic-MW-block}.

The coefficient realizations retain different parts of this
formula.  In Milnor $K$-theory, $\eta=0$ and $g=1$, so the differential
is left symbol contraction without an additional incidence sign.
Under normalized real signature, $g$ acts by $-1$ and a framed
contraction is the difference of the two normal-side values.
Thus \eqref{eq:full-cellular-differential} gives precisely
\eqref{eq:sign-differential}.  Rational Hodge realization kills both
$\eta$ and $[-1]$, whereas signature records their real values.
Integral Milnor $K$-theory retains the symbol $\{-1\}$.

\subsection{The last Milnor \texorpdfstring{$K$}{K}-theory differential}
Under \AS, the differential into degree $p$ of the cellular cochain complex with coefficients $\KM_p$ has an explicit right inverse for every $p>0$. We construct it using the quadratic frame convention fixed above. Put
\[
C_M^{\Sigma,\bullet}(X,p)=\Hom_{\calA}(C_*^\Sigma(X),\KM_p).
\]
Its degree-$r$ term is $\bigoplus_{|I|=r}H^0(\mathcal O_I,\KM_{p-r})$. We retain the normal frames from the ordered mutable coordinates.

\begin{proposition}\label{prop:acyclic-chow}
Suppose that $k$ is perfect of characteristic different from two, and let $X$ be the acyclic cluster variety of a seed satisfying \AS. For every $p>0$,
\begin{align*}
C_M^{\Sigma,p-1}(X,p)&=\bigoplus_{|I|=p-1}\mathcal O(\mathcal O_I)^\times,\\
C_M^{\Sigma,p}(X,p)&=\bigoplus_{|J|=p}\Z e_J,\qquad
C_M^{\Sigma,r}(X,p)=0\quad(r>p).
\end{align*}
For each independent set $J$ of size $p$, choose $j\in J$ and put $I=J\setminus\{j\}$. With the convention of Theorem~\ref{thm:full-differential}, one has
\begin{equation}\label{eq:explicit-chow-residue}
d^{p-1}\{x_j\}_I=e_J.
\end{equation}
Here $\{x_j\}_I$ is supported in the summand indexed by $I$. Thus $e_J\mapsto\{x_j\}_I$ defines a right inverse to $d^{p-1}$, and $\CH^p(X)=0$.
\end{proposition}
\begin{proof}
The displayed terms follow from $\KM_1=\Gm$, $\KM_0=\Z$, the vanishing of negative Milnor $K$-theory, and geometric integrality of the strata. Since $j\notin I$, the coordinate $x_j$ is invertible on $\mathcal O_I$. The cellular differential on this summand is the divisor residue along adjacent strata $\mathcal O_{I\cup\{\ell\}}$. The simple normal-crossings coordinates give
\[
\operatorname{ord}_{\mathcal O_{I\cup\{\ell\}}}(x_j)
=\begin{cases}1,&\ell=j,\\0,&\ell\ne j.\end{cases}
\]
In the second case, $x_j$ remains invertible on the target stratum. All components of the differential whose target does not contain $I$ are zero. Moreover, the quadratic orientation factor in \eqref{eq:full-cellular-differential} acts as one on Milnor $K$-theory. These observations give \eqref{eq:explicit-chow-residue} with the indicated positive sign. Extending the assignment $\Z$-linearly gives a right inverse and hence proves split surjectivity. Finally, the Bloch formula $\CH^p(X)=H^p_{\Nis}(X,\KM_p)$, applied to the anticlique resolution, proves the vanishing assertion.
\end{proof}

Since RFR implies \AS, this calculation recovers the acyclic case of
Hlavinka's Chow vanishing theorem \cite[Theorem~5.5]{Hlavinka}.
For the full RFR sink-recurrent class, the freezing--deletion decomposition
\cite[Propositions~4.13 and~4.16]{Hlavinka} and Chow localization give the
same vanishing by induction: higher codimensions follow from the smaller
pieces, while the remaining codimension-one generator is the principal
divisor of the sink variable.

With coefficients $\KMW_p$, the degree-$p$ term is
$\bigoplus_{|J|=p}H^0(\mathcal O_J,\KMW_0)$, and the terms above
degree $p$ involve Witt sheaves. The Milnor--Witt differential above
retains these quadratic terms. Appendix~\ref{sec:witt} computes its
Witt specialization and the resulting positive-degree Chow--Witt
groups, applied in Section~\ref{sec:sevenfold}.

\subsection{The oriented Hodge realization}
We continue to assume \AS\ and allow the principal matrix to be skew-symmetrizable.
The purity construction gives a Gysin complex, which in the skew-symmetric
RFR case recovers the complex of Lam--Speyer \cite{LS2}.
We describe its residue convention and derive a consequence for contractibility.
The stabilized graded purity piece at level $p$ is
\begin{equation}\label{eq:graded-purity}
\bigvee_{|I|=p} S^{2p,p}\wedge(\mathcal O_I)_+
\simeq\bigvee_{|I|=p}\ \bigvee_{a=0}^{d-2p}
(S^{2p+a,p+a})^{\vee\binom{d-2p}{a}}.
\end{equation}

\begin{theorem}\label{thm:gysin}
For $k\subseteq\C$, the Hodge realization of the same purity filtration gives
\[
E_1^{p,q}=\bigoplus_{|I|=p}H^{q-p}(\mathcal O_I,\Q)(-p)
\ \Longrightarrow\ H^{p+q}(X,\Q).
\]
For skew-symmetric RFR seeds this is the Gysin spectral sequence of Lam--Speyer \cite{LS2}; the same construction applies under \AS. Its $d_1$ is the adjacent-stratum residue and $E_2=E_\infty$. Write $G_{\mathrm{Gys}}^{p,s}=\bigoplus_{|I|=p}H^{s-p}(\mathcal O_I,\Q)(-p)$ with this residue differential. Consequently,
\[
\gr^W_{2s}H^j(X,\Q)\simeq H^{j-s}(G_{\mathrm{Gys}}^{\bullet,s}).
\]
\end{theorem}
\begin{proof}
Purity for support in a complex codimension-$p$ smooth stratum shifts cohomology by $2p$ and twists by $-p$, giving the stated first page. Since $H^{q-p}(\mathcal O_I)$ has weight $2(q-p)$, every summand has weight $2q$. A differential $d_r$ has target weight $2(q-r+1)$ and therefore vanishes for $r\geq2$ by strictness of the weight filtration.

The $d_1$ maps are the following residues, agreeing with
\cite[Propositions 3.11 and 5.6]{LS2} in their skew-symmetric setting.
Put $\alpha_i=\sum_a\Btil_{ai}\,d\log x_a$. The cohomology of $\mathcal O_I$ is the exterior algebra on $d\log x_a$, $a\notin I$, modulo the relations $\alpha_i=0$, $i\in I$. If $J=I\cup\{j\}$ is independent, the differential is given by left contraction:
\[
\theta=\theta_1+d\log x_j\wedge\theta_2\longmapsto\theta_2
\quad\text{in the quotient for }J.
\]
Here $\theta_1$ and $\theta_2$ do not contain $d\log x_j$. Independence gives $\widetilde B_{ji}=0$ for $i\in I$, so contraction descends through the existing relations. The contractions anticommute, including after the further quotient, and hence give a differential. Lam--Speyer use right extraction; in bidegree $(p,s)$ their differential differs from this one by $(-1)^{s-p-1}$. Multiplication in degree $p$ by $(-1)^{p(s-1)-p(p-1)/2}$ identifies the two complexes. This proves the comparison and fixes the matrix convention used below.
\end{proof}

\begin{corollary}\label{cor:top-weight}
If $X/\C$ is an acyclic cluster variety satisfying \AS\ of dimension
$d>0$, then $\gr^W_{2d}H^d(X(\C),\Q)=\Q(-d)$.
Consequently $X$ is neither topologically contractible nor stably
$\A^1$-contractible, even after rationalization.
\end{corollary}
\begin{proof}
In the row $q=d$ of the spectral sequence, the initial torus gives
$E_1^{0,d}=\Q(-d)$. For $p>0$, a stratum torus has dimension
$d-2p<d-p$, so $E_1^{p,d}=0$. Theorem~\ref{thm:gysin} proves
the asserted weight-graded group. Its nonzero reduced rational
cohomology obstructs topological contractibility and, by complex
Betti realization \cite{LevineBetti}, stable motivic contractibility.
\end{proof}

\subsection{The square of a cluster fourfold}
We return to the fourfold in Lam--Speyer
\cite[Section~6.3, Table~2, left]{LS1}.
\begin{lemma}\label{lem:LS-square}
Let $X$ be \eqref{eq:LS-fourfold}, over a perfect field of characteristic
different from two. For every strictly $\A^1$-invariant sheaf $M$,
\[
 H^1_{\Nis}(X,M)\simeq M_{-2}(k),\qquad
 H^2_{\Nis}(X\times X,M)\simeq M_{-4}(k).
\]
Moreover $K_0(X\times X)=\Z$ and
$\operatorname{cd}_{\A^1}(X\times X)=2$.
The two frozen coordinates act trivially on
$\CW^2(X\times X)\simeq\Witt(k)$.
\end{lemma}
\begin{proof}
The extended matrix has a unit minor in rows $2,3,4$, so the seed
is RFR. The three codimension-one strata have torus coordinates
$(a,b)=(x_2,x_3),(x_1,t),(x_1,t)$, respectively.
Theorem~\ref{thm:full-differential} computes the boundary from the maps
\[
 (u,a,b)\longmapsto(u,a,b,-ab),\quad
 (u,a,b)\longmapsto(a,u,-a,b),\quad
 (u,a,b)\longmapsto(a,-a^{-1},u,b),
\]
with target coordinates $(x_1,x_2,x_3,t)$.
We index the summands of $C_1$ by $(j,A)$, where
$j=1,2,3$ and $A=\varnothing,a,b,ab$, and those of $C_0$
by subsets of $\{1,2,3,4\}$, with $4$ denoting $t$.
To simplify the transpose boundary, we cancel $(j,\varnothing)$ against column
$j$, then $(2,b),(3,b)$ against columns $24,34$, and finally
$(1,ab),(2,ab),(3,ab)$ against columns $123,124,134$.
The respective pivots are $1,1,1,1,1,1,-g,-g$, where
$g=\la-1\ra$. The only remaining nonzero block, on rows
$(1,a),(1,b),(2,a),(3,a)$ and columns $12,13,14,23$, is
\[
 \begin{pmatrix}
 1&0&g&0\\
 0&1&g&0\\
 -g&0&0&g\\
 0&-g&0&g
 \end{pmatrix}.
\]
Canceling its first two pivots leaves columns $(1,1)^{\mathsf T}$
and $g(1,1)^{\mathsf T}$, so one further unit cancellation gives
\[
 C_*^\Sigma(X)\simeq
 \left(\Z\oplus\bigoplus_{r=1}^4\KMW_r\right)[0]
 \oplus\KMW_2[1].
\]
All cancellations respect the Milnor--Witt degrees and define
chain homotopies. The surviving $\KMW_2$ embeds into the original
$C_1$, in the order just specified, by
\[
 v=(0,-g,g,0;\ 0,-1,0,0;\ 0,1,0,0)^{\mathsf T}.
\]
The cellular product formula \cite[Lemma~2.31]{MS} tensors these
homotopies. Since
$\KMW_2\otimes_{\A^1}\KMW_2\simeq\KMW_4$
\cite[Remark~2.45]{MS}, applying $\Hom_{\calA}(-,M)$ gives the
two cohomology formulas and the upper bound on cohomological dimension.
Taking $M=\KMW_2$ gives the nonzero group $\Witt(k)$ and hence equality
in the bound. The assertion about stable classes follows by applying
Proposition~\ref{prop:kzero} to the product seed.

It remains to determine the action of units. Write $(a,b)=(x_1,t)$ and $(c,d)=(y_1,s)$ on the third stratum
in the two factors. Under the chosen splitting, a class
$\alpha\in\Witt(k)$ is represented there by $[a][c]\alpha$.
Multiplication by $\la t\ra=1+\eta[b]$ adds only a component
involving $ab$ in the first factor, which evaluates to zero on
$v\otimes v$ because $v_{(3,ab)}=0$; the same argument applies to
$s$. Finally the units are $k^\times t^{\Z}s^{\Z}$: on the initial
torus every unit is a Laurent monomial, and valuation along each
mutable divisor kills its corresponding exponent. Thus the full
orientation action is scalar multiplication by $\la\lambda\ra$.
\end{proof}

\section{Integral contractions for paths}
\label{app:path-vanishing}

In this appendix we prove the coefficient vanishing used in
Theorem~\ref{thm:ordered-freezing-triviality} by constructing integral chain homotopies.
These take place in the category of strictly $\A^1$-invariant
sheaves. We write $K_j=\mathbf K_j^{MW}$ for $j>0$; the constant
summand in weight zero is $\Z$.

\begin{theorem}\label{thm:integral-path-vanishing}
Let $k$ be perfect of characteristic different from two. Suppose
an acyclic exchange presentation satisfies \AS, has all its $m$
frozen variables inverted, and its principal matrix is a disjoint
union of paths with single arrows. Nonzero scalar coefficients
are allowed in the exchange equations. Then its variety $X$ has
$\A^1$-cohomological dimension zero.

More precisely, if the connected components have $n_1,\ldots,n_c$
vertices and $q$ of the $n_i$ are odd, put
\[
 P(z)=(1+z)^{m-q}
 \prod_{n_i\ {\rm even}}(1+z^2+\cdots+z^{n_i})
 \prod_{n_i\ {\rm odd}}(1+z+\cdots+z^{n_i+1})
       =\sum_j b_jz^j.
\]
Then $m\ge q$, and a cellular chain complex of $X$ has an
integral chain-homotopy equivalence
\[
 C_*^\Sigma(X)\simeq
 \left(\Z\oplus\bigoplus_{j>0}K_j^{\oplus b_j}\right)[0].
\]
\end{theorem}

\begin{proof}
We first consider the linearly oriented presentations
\[
 x_i y_i=x_{i-1}+x_{i+1}\quad(1\le i\le n).
\]
For $n=2a$ and $x_0=x_{2a+1}=1$, denote the variety by $Y_{2a}$.
For arbitrary $n$, take $x_0=t\in\Gm$ and $x_{n+1}=1$, and
denote the resulting variety by $Z_n$. Set $Y_0=\operatorname{Spec}k$
and $Z_0=\Gm$. We establish
\begin{equation}\label{eq:minimal-path-contractions}
 \begin{split}
 C_*^\Sigma(Y_{2a})&\simeq
  (\Z\oplus K_2\oplus\cdots\oplus K_{2a})[0],\\
 C_*^\Sigma(Z_n)&\simeq
  (\Z\oplus K_1\oplus\cdots\oplus K_{n+1})[0].
 \end{split}
\end{equation}

To keep track of the contractions, we give $K_j$ weight $j$ and write
\[
 g=\langle-1\rangle,\quad \rho=[-1],\quad\epsilon=-g,
 \qquad R=\Z[g,\rho]/(g^2-1,(1+g)\rho),
 \qquad |g|=0,\quad |\rho|=1.
\]
This ring acts through Milnor--Witt theory, with
$g=1+\eta\rho$. On a stratum $\mathcal O_I$, the exchange equations
at vertices $i\in I$ identify the nonzero $x$-coordinates by
$x_{i+1}=-x_{i-1}$ or set them to $\pm1$.
We choose one coordinate from each equivalence class not fixed to a constant.
Every remaining $x$-coordinate is then a signed chosen coordinate or a constant.
For a mutable vertex $j\notin I$ with $I\cup\{j\}$ independent,
the class of $x_j$ is a singleton, since $j$ is not adjacent to any vertex of $I$.
Thus $x_j$ itself can be used as the normal parameter for this boundary. Theorem~\ref{thm:full-differential}, together with
\[
 [-u]=\rho+g[u],\qquad [u]^2=\rho[u],\qquad
 [u][v]=\epsilon[v][u],
\]
shows that the complete cellular matrices have entries in $R$,
homogeneous of degree $l-j$ from weight $j$ to weight $l$.
In particular they do not lower weight. Products of distinct
free coordinates are not used in these minimal coordinates;
the $\eta[u][v]$ term for such a product is therefore absent.
The signed restriction maps commute around each two-face, and
$g+\epsilon=0$ supplies cancellation of the normal-order factors.
Thus the chain identity holds already over the formal ring $R$.

We begin by computing the cokernel of the weight-preserving first boundary
over $A=\Z[g]/(g^2-1)$. The change of signs
$x_i=(-1)^{\lfloor i/2\rfloor}z_i$ makes each interior zero
equation identify $z_{i-1}$ with $z_{i+1}$.
For $J=\{j_1<\cdots<j_r\}$, write $e_J=e_{j_1}\cdots e_{j_r}$ for the square-free symbol corresponding to $[z_{j_1}]\cdots[z_{j_r}]$, with $e_\varnothing=1$.
The singleton stratum at an interior mutable vertex $i$ has
one common neighbor coordinate. Omitting or selecting it gives,
up to an invertible common factor, the relations
\[
 e_i e_J=0,\qquad
 \epsilon e_{i-1}e_i e_J+e_i e_{i+1}e_J=0,
 \qquad J\cap\{i-1,i,i+1\}=\varnothing.
\]
At an endpoint with constant outer coordinate only the first
relation occurs. In $Z_n$, the initial coordinate $z_0$ is frozen
and has no boundary relation. These are all the relations, since
the boundary into degree zero comes precisely from singleton
strata. Repeated symbols contribute a factor $\rho$ and hence
no additional weight-preserving relation.

We encode $e_J$ by a binary word whose entry at position $i$ is $1$ if $i\in J$ and $0$ otherwise. The second relation slides
$110$ to $011$ with coefficient $g$, and replacing $e_J$ by
$g^{\lfloor(\sum_{j\in J}j)/2\rfloor}e_J$ makes the coefficient of every slide
equal to one. A word with $v>0$ vacancies has a unique expression
\[
 1^{a_0}0\,1^{a_1}0\cdots0\,1^{a_v},\qquad a_i\ge0,
\]
including empty runs at both ends and between successive vacancies.
A slide decreases one run length by two and increases an adjacent run length
by two. At fixed total weight, two words are connected by slides exactly when
their run-length vectors have the same parity: after subtracting these parities
and dividing by two, adjacent transfers connect all nonnegative vectors with
the same sum. If an odd run contains only mutable positions, transfers can
reduce it to a singleton, whose symbol is killed by the first relation. For $Y_{2a}$ this kills
every odd run, including both ends. For $Z_n$, only the initial
run can be odd without being killed: it can be reduced to the
single frozen position $0$. Thus the sole possible surviving
parity vector is all zero for $Y_{2a}$, and has only $a_0$
possibly odd for $Z_n$.
It follows that one class survives in every even weight for $Y_{2a}$ and
in every weight for $Z_n$, with the fully selected word supplying
the top class. Since each relation either identifies classes by a unit or kills
an entire component, the cokernels are free over $A$:
\begin{equation}\label{eq:path-first-cokernels}
 \operatorname{coker}d_1^0(Y_{2a})=A\{0,2,\ldots,2a\},
 \qquad
 \operatorname{coker}d_1^0(Z_n)=A\{0,1,\ldots,n+1\}.
\end{equation}
Here the braces record weights. The formal weight-zero $A$ is replaced
by $\Z$ in the sheaf complex and never participates in a pivot.

We now construct the full chain homotopies by induction. The
case $Z_0$ is the torus splitting. For $Z_1$ the open subset $D(x_1)$ is
$\Gm^2$ and the closed divisor is $\A^1$ with $t=-1$.
The Thom $K_1$ has an identity component in the $x_1$ target,
so it cancels and leaves $\Z\oplus K_1\oplus K_2$.
For $n\ge2$, freezing the first vertex gives
\[
 D(x_1)\cong Z_{n-1}\times\Gm^{(t)},\qquad
 V(x_1)\cong\A^1\times Z_{n-2}.
\]
The endpoint parameter for the $Z_{n-1}$ factor is $x_1$, whereas
for the $Z_{n-2}$ factor it is $-t=x_2$; $y_1$ is the free affine
coordinate. These identifications use only signed coordinates.
In the induction, we require the contractions to have entries in $R$
and not to lower weight. Tensoring with the torus and applying
the Thom shift preserve these properties. The contractions obtained
by induction therefore replace the cellular cone by an actual two-term
complex over $R$, concentrated in homological degrees one and zero.
Its weight-preserving quotient remains chain-homotopy equivalent to
that of the original complex.

The degree-one term has one summand in each weight $1,\ldots,n$.
The degree-zero term has two summands in each of these weights and
one summand in each of weights zero and $n+1$.
By \eqref{eq:path-first-cokernels}, each weight-preserving column
$A\to A^2$ has cokernel $A$ and is a split injection. Indeed a
quotient row $(a,b)$ is unimodular, its kernel is generated by
$(-b,a)$, and the original column surjects onto this kernel;
its scalar coefficient in that generator must be a unit.
Using these pivots, we obtain a full square block equal to its
invertible weight-preserving part plus a strictly weight-raising
matrix. Multiplying by the inverse of the weight-preserving part gives
a matrix of the form $1+N$, with $N$ nilpotent. A finite geometric series
then gives an inverse to the block.
Cancelling this block leaves precisely one target in each weight.
The resulting homotopies again belong to $R$ and do not lower
weight, proving the assertion for $Z_n$.

For $Y_{2a}$ the open complement and closed divisor are, respectively,
$Z_{2a-1}$ and $\A^1\times Y_{2a-2}$, after constant sign changes.
The contracted cone has one source in each odd weight
$1,3,\ldots,2a-1$ and one target in each weight $0,\ldots,2a$.
Equation~\eqref{eq:path-first-cokernels} makes every
weight-preserving odd-weight scalar a unit. The same finite
triangular cancellation proves \eqref{eq:minimal-path-contractions}.
No specialization of $g$ or $\rho$ has been made.

It remains to reduce general orientations and frozen coefficients.
For a path forest, we choose bipartitions $P,Q$, with $P$ larger
on each odd component. The matrix $B_{Q,P}$ has full row rank
and a unimodular maximal minor, supplied by the consecutive
matching along each path. Integral basis changes put it into
$(I\ 0)$. Eliminating these unit pivot blocks in the full extended
matrix leaves exactly the frozen matrix on the one residual
column of each odd component. The same residual matrix occurs
when only the independent column family $P$ is retained.
Condition \AS\ makes it primitive and independent. Thus the
whole extended matrix is RFR.

We first remove the nonzero scalar coefficients. The character
homomorphism with exponent columns
$\widetilde B$ is split surjective, so the prescribed exchange
ratios have a $k$-rational preimage. Rescaling mutable and frozen
coordinates by that preimage, and the exchange partners by
the resulting common constants, makes both coefficients one.
This requires no extraction of roots. Source and sink mutations
then orient each path linearly \cite[Proposition~9.2]{FZII}.
Such mutations preserve RFR and the cluster variety.

The mutable row lattice of an even linear path is all of
$\Z^{n_i}$, whereas for an odd path it is the saturated kernel
of pairing with $(1,0,1,\ldots,1)$. The quotient for the forest
is therefore $\Z^q$. RFR says that the frozen rows generate this quotient.
Let $C_{\mathrm{std}}$ be the frozen matrix with one endpoint coefficient
for each odd component, as in $Z_{n_i}$, and zero rows for the other frozen
coordinates. We choose frozen coordinates so that the rows of the
transformed matrix $C_1$ have the same images in the quotient as the
corresponding rows of $C_{\mathrm{std}}$. Thus
$C_1=C_{\mathrm{std}}+UB$ for an integral matrix $U$. Let $u_a$ be
the columns of $U$. Substituting $x_a=t^{-u_a}X_a$ in each exchange
equation transforms the two frozen exponent vectors into
\[
 a_i=(C_1)_{+,i}-\sum_a[B_{ai}]_+u_a,\qquad
 b_i=(C_1)_{-,i}-\sum_a[-B_{ai}]_+u_a.
\]
Their difference is $(C_{\mathrm{std}})_i$. With the
coordinatewise minimum $\ell_i=\min(a_i,b_i)$, set
$Y_i=t^{-u_i-\ell_i}y_i$. This gives the standard
exchange equation for $C_{\mathrm{std}}$. We normalize a frozen torus basis
change in the same way, by factoring out the common
frozen monomial in its two exchange terms. Both directions
multiply mutable coordinates and their partners only by frozen
units, and therefore remain regular on all zero strata.
We obtain an actual isomorphism
\[
 X\cong\prod_{n_i\ {\rm even}}Y_{n_i}\times
         \prod_{n_i\ {\rm odd}}Z_{n_i}\times\Gm^{m-q}.
\]
This change of coordinates is used only after the minimal
contractions are proved; it need not preserve their weights
or frames. The cellular product formula
\cite[Lemma~2.31]{MS} tensors the actual chain homotopies
and gives the polynomial $P(z)$. This tensoring of chain homotopies
does not require flatness of projective strictly invariant sheaves.
Finally, applying $\operatorname{Hom}(-,M)$ proves the required vanishing
for every strictly invariant coefficient sheaf $M$.
\end{proof}
\section{Connected components, Chow groups, and stable freeness}\label{sec:component-chow}
Ordinary Chern classes give the first obstruction to stable freeness.
We therefore compute the Chow groups before passing to quadratic coefficients.
Throughout this appendix, the seed is acyclic and satisfies \AR\ over
an arbitrary field. The finite set $\pi_0(T_I)$ denotes the
connected components over that field, rather than geometric points.

\subsection{Maps on connected components}

Let $J=I\cup\{j\}$ be independent. Since $\Btil_{ji}=0$ for all $i\in I$,
the coordinate $x_j$ is a free torus factor:
\[
 T_I=(\Gm)_{x_j}\times S_{I,j}.
\]
The additional equation $\chi_j=-1$ embeds $T_J$ into $S_{I,j}$. It
therefore induces a map on finite \'etale component schemes and on their
sets of connected components over $k$:
\[
 \pi_0(T_J)\longrightarrow\pi_0(T_I).
\]
Write
\begin{equation}\label{cc:eq:component-pullback}
 R_{I,J}:\Z^{\pi_0(T_I)}\longrightarrow\Z^{\pi_0(T_J)}
\end{equation}
for the pullback of integer-valued functions. The column indexed by
$C\in\pi_0(T_I)$ is the indicator function of the inverse image of $C$
in $\pi_0(T_J)$, so its entries are $0$ and $1$. Since these maps pull back
component functions rather than push forward cycles, no extension-degree
factors occur.

\begin{theorem}\label{cc:thm:chow}
Assume \AR. For every $p>0$, there is an isomorphism
\begin{equation}\label{cc:eq:chow-formula}
 \CH^p(X)\cong
 \bigoplus_{\substack{J\in\Ind(\Gamma)\\|J|=p}}Q_J,
 \qquad
 Q_J=\operatorname{coker}\left(
 \bigoplus_{j\in J}\Z^{\pi_0(T_{J\setminus\{j\}})}
 \xrightarrow{\ \sum_jR_{J\setminus\{j\},J}\ }
 \Z^{\pi_0(T_J)}\right).
\end{equation}
Moreover, $\CH^0(X)=\Z$ and $\CH^p(X)=0$ for $p>\alpha(\Gamma)$, where
$\alpha(\Gamma)$ is the largest cardinality of an independent set.
\end{theorem}

For $p=1$ the formula is
$\Pic(X)=\bigoplus_i\Z^{|\pi_0(T_{\{i\}})|-1}$.
In characteristic zero this case agrees with the divisor-class-group
calculation of Garcia Elsener--Lampe--Smertnig
\cite[Theorems A and 5.5]{GELS}; the result above also determines the Chow groups
in all higher codimensions.

\begin{proof}
We work over the arbitrary ground field with Rost's Milnor $K$-theory
cycle complex. For a smooth equidimensional $Y$, write
\[
C^a(Y;K^M_*,s)=\bigoplus_{y\in Y^{(a)}}K^M_{s-a}(\kappa(y)),
\qquad A^a(Y;K^M_*,s)=H^a C^\bullet(Y;K^M_*,s).
\]
Its differential is given by residues and norms. In particular,
$A^p(X;K^M_*,p)=\CH^p(X)$ by the definition of rational equivalence;
see \cite[Section 5 and Remark 5.1]{Rost}.
Localization and affine-bundle invariance
\cite[Section 5 and Proposition 8.6]{Rost} give
\begin{equation}\label{cc:eq:acyclic}
A^q(\mathcal O_I;K^M_*,s)=0\qquad(q>0)
\end{equation}
for every $s$. Indeed, the vanishing holds for a field and hence for each finite
separable field occurring in $E_I$. Affine factors do not change these
groups, and localization for $Y\times\{0\}\subset Y\times\A^1$
proves the assertion inductively for split torus factors. The same
vanishing then holds for their finite disjoint unions.

We set $\Omega_{-1}=\varnothing$, put $Z_r=X\setminus\Omega_{r-1}$,
and filter $C^\bullet(X;K^M_*,p)$
by cycles supported on $Z_r$. The associated graded complex in
filtration degree $r$, reindexed by $q=a-r$, is
\[
\bigoplus_{|I|=r}C^q(\mathcal O_I;K^M_*,p-r).
\]
Its cohomology is zero for $q\ne0$. Thus the filtration spectral
sequence has a single row, with no higher differentials or extension
ambiguities. Its row complex is
\begin{equation}\label{cc:eq:cochains}
C^r(X,p)=\bigoplus_{\substack{I\in\Ind(\Gamma)\\|I|=r}}
A^0(T_I;K^M_*,p-r),
\qquad H^p(C^\bullet(X,p))=\CH^p(X).
\end{equation}
The cycle-module calculation is valid over an arbitrary base field
under \AR.

To compute the last differential, note that there are no terms above degree $p$
and that the final two terms are
\[
 C^{p-1}=\bigoplus_{|I|=p-1}\mathcal O(T_I)^\times,
 \qquad
 C^p=\bigoplus_{|J|=p}\Z^{\pi_0(T_J)}.
\]
The differential is the divisor residue between adjacent strata. There
is no quadratic orientation factor: the Milnor--Witt orientation action
maps to $1$ in Milnor $K$-theory. Ordinary incidence sign conventions, if
used, do not change any of the images below.

For a fixed $I$, let $A(I)$ be the set of mutable vertices outside $I$ with
no edges to $I$. None of the characters $\chi_i$, $i\in I$, involves
any coordinate $x_j$ with $j\in A(I)$. There is therefore a simultaneous
factorization
\[
 T_I=\Gm^{A(I)}\times Y_I.
\]
Each connected component $C$ of $Y_I$ is a split torus over a finite
separable field extension of $k$. Consequently, every unit on
$\Gm^{A(I)}\times C$ has a unique expression
\begin{equation}\label{cc:eq:unit-expression}
 u_C\prod_{j\in A(I)}x_j^{a_{C,j}},
 \qquad u_C\in\mathcal O(C)^\times,\quad a_{C,j}\in\Z.
\end{equation}
At the boundary indexed by $J=I\cup\{j\}$, the factor $u_C$ and all
$x_\ell$ with $\ell\ne j$ remain invertible. The residue of
\eqref{cc:eq:unit-expression} equals $a_{C,j}$ on each component of
$T_J$ whose image in $Y_I$ lies in $C$.

The integers $a_{C,j}$ are independent for different $C$ and $j$.
Equivalently, the unit on $T_I$ equal to $x_j$ on
$\Gm^{A(I)}\times C$ and to $1$ on all other components has residue
at $J$ equal to the indicator function of the components of $T_J$
mapping to $C$, and zero at every other stratum of codimension $p$.
Thus the image of
$C^{p-1}\to C^p$ is the direct sum, over $J$, of the images appearing in
\eqref{cc:eq:chow-formula}. In particular, no relation mixes distinct $J$,
which proves the asserted cokernel and direct-sum formulas. The remaining
claims follow from the dense initial torus and the absence of strata
in codimension greater than $\alpha(\Gamma)$.
\end{proof}

\subsection{Connected components and stable classes}

\begin{theorem}\label{cc:cor:fixed}
Under \AR, the following are equivalent:
\begin{enumerate}
\item $\CH^p(X)=0$ for every $p>0$;
\item $T_I$ is connected over $k$ for every $I\in\Ind(\Gamma)$;
\item every finite \'etale algebra defining $E_I$ is a field;
\item the rank map $K_0(X)\to\Z$ is an isomorphism;
\item every algebraic vector bundle on $X$ is stably free.
\end{enumerate}
\end{theorem}

\begin{proof}
If all strata are connected, then for every nonempty independent set $J$,
the target in the presentation of $Q_J$ is $\Z$, and any one predecessor
map takes $1$ to $1$. Hence all cokernels
vanish. Conversely, choose an independent $J$ of minimal cardinality
for which $c_J=|\pi_0(T_J)|>1$. All its proper predecessors are connected,
so each map into $\Z^{c_J}$ has the same diagonal image. Therefore
\[
 Q_J=\Z^{c_J}/\Z(1,\ldots,1)\cong\Z^{c_J-1}\ne0.
\]
Theorem~\ref{cc:thm:chow} identifies this group with a direct summand of
$\CH^{|J|}(X)$, rather than merely a quotient or an associated graded group.
The equivalence with the third condition follows from
\eqref{cc:eq:smith-stratum}.

Suppose now that every $T_I$ is connected. Then
$\mathcal O_I\cong(\A^{|I|}\times\Gm^{d-2|I|})_{L_I}$ for a finite
separable field $L_I/k$, so $G_0(\mathcal O_I)=\Z$, generated by its
structure sheaf. For nonempty $I$, the regular sequences $(x_i)_{i\in I}$
give zero pushforward in $G_0$, exactly as in the Koszul and
localization argument of Proposition~\ref{prop:kzero}. Induction
along the filtration gives $K_0(X)=\Z$ by rank, and the
projective-complement argument there proves stable freeness.
Conversely, stable freeness of all vector bundles
implies that the rank map on $K_0$ is an isomorphism.

Finally, if a stratum is disconnected, choose $J$ of smallest
cardinality as above. The resulting nonzero free summand of
$\CH^{|J|}(X)$ survives rationalization. The rational Chern character
isomorphism for smooth quasi-projective varieties
\cite[Corollary 18.3.2]{Fulton} contradicts $K_0(X)=\Z$.
Thus all five conditions are equivalent.
\end{proof}

\begin{corollary}\label{cc:cor:universal}
Under \AR, the following are equivalent:
\begin{enumerate}
\item $\CH^p(X_K)=0$ for every $p>0$ and every field extension $K/k$;
\item $\CH^p(X_{\bar k})=0$ for every $p>0$;
\item condition \AS\ holds;
\item every vector bundle on $X_K$ is stably free for every field
extension $K/k$.
\end{enumerate}
It suffices to check \AS\ for maximal independent sets.
\end{corollary}

\begin{proof}
Over $\bar k$, the finite \'etale scheme $E_I$ has exactly
$\prod_\ell e_\ell$ connected components. Thus all strata are geometrically
connected if and only if every Smith factor equals $1$.
Theorem~\ref{cc:cor:fixed} proves the equivalence, since \AS\ is unchanged
by field extension. Every subset of a saturated independent collection
of columns remains saturated. Passing from a maximal independent set to
a subset restores only rows that are zero on the retained columns.
Applying Theorem~\ref{cc:cor:fixed} over each extension gives the stable-freeness criterion, and the preceding observation proves the final assertion.
\end{proof}

\begin{corollary}\label{cc:cor:first}
Suppose \AR\ holds and $k$ is algebraically closed. For each independent
$I$, let
\[
 \Delta_I=\prod_\ell e_\ell
 =\gcd\{\text{maximal minors of }B_I\},
 \qquad \Delta_\varnothing=1.
\]
If some $\Delta_I>1$, put
$r=\min\{|I|:\Delta_I>1\}$. Then
\begin{equation}\label{cc:eq:first-nonzero}
 \CH^p(X)=0\quad(0<p<r),\qquad
 \CH^r(X)\cong
 \bigoplus_{\substack{I\in\Ind(\Gamma)\\|I|=r}}
 \Z^{\Delta_I-1}.
\end{equation}
\end{corollary}

\begin{proof}
All strata of codimension less than $r$ are connected, so their cokernels
vanish. If $|I|=r$, every predecessor is connected and the target has
$\Delta_I$ components. Its relation subgroup is the diagonal copy of
$\Z$, so the result follows from Theorem~\ref{cc:thm:chow}.
\end{proof}

\begin{remark}\label{cc:rem:AR-scope}
Beyond the class satisfying \AR, Chow vanishing also occurs for the
coefficient-free rank-one seed $B=(0)$ over a field of characteristic
zero: its equation is $xx'=2$, so $X\cong\Gm$.
\end{remark}

\begin{corollary}\label{cor:AS-products}
Under \AS, $X$ is linear and, for every finite-type $k$-scheme $S$,
flat pullback is an isomorphism
\[
\CH_j(S)\xrightarrow{\sim}\CH_{j+d}(X\times S).
\]
\end{corollary}
\begin{proof}
Each anticlique stratum is an affine space times a split torus and is therefore
linear. It follows that their finite stratified union $X$ is also linear.
Theorem~\ref{cc:thm:chow} gives $\CH_*(X)=\Z[X]$.
Totaro's integral Chow K\"unneth theorem for linear schemes
\cite[Section 3, Proposition 1]{TotaroL} gives the assertion.
\end{proof}

Suppose some $T_I$ is disconnected over $k$, and let $r$ be the least
cardinality of such an independent set $I$. The proof of
Theorem~\ref{cc:cor:fixed} produces a vector bundle that is not stably
free. More concretely, choose such an $I$ with $|I|=r$, and let $Z$
be the closure in $X$ of one connected component of $\mathcal O_I$.
Then $[\mathcal O_Z]\in K_0(X)$ has rank zero and satisfies
$\operatorname{ch}_r([\mathcal O_Z])=[Z]\ne0$ rationally.
A finite projective resolution and a projective complement represent
this class as $[E]-\rk(E)[\mathcal O_X]$ for a vector bundle $E$.
Since all lower positive Chow groups vanish,
\[
c_r(E)=(-1)^{r-1}(r-1)![Z]\ne0
\]
in the free group $\CH^r(X)$; see \cite[Example 15.3.6]{Fulton}.

\subsection{Beyond the saturated and full-rank cases}\label{cc:sec:examples}
The following coefficient-free matrix satisfies \AS\ and has
deficient column rank:
\[
 \begin{pmatrix}0&1&1\\-1&0&1\\-1&-1&0\end{pmatrix}.
\]
It has rank two, but its only nonempty independent sets are singletons
and its columns are primitive. Over a perfect field of characteristic
different from two, all vector bundles on its smooth cluster threefold are
free by Corollary~\ref{cor:cochordal-triviality}.
Primitivity of individual columns, however, is insufficient.
For two isolated mutable vertices and frozen block
$\left(\begin{smallmatrix}1&1\\1&-1\end{smallmatrix}\right)$,
in characteristic different from two, the singleton strata are
connected and the joint stratum has two rational components.
The component formula gives $\CH^1=0$ and $\CH^2=\Z$, so not every
vector bundle is stably free.

Connectedness over the ground field must also be distinguished
from geometric connectedness. For the coefficient-free
$2$-Kronecker surface over $\Q$, each singleton stratum is
$\A^1_{\Q(\sqrt{-1})}$, with
$T_{\{i\}}\cong\Spec\Q(\sqrt{-1})$.
Thus $K_0=\Z$ and $\CH^{>0}=0$ over $\Q$,
whereas over $\bar\Q$ the Picard group is $\Z^2$.

\subsection{Component groups over an algebraic closure}
\label{cc:sec:lattice}
Over $\bar k$, the component maps in the Chow formula have the
following lattice description. For each independent set $I$, put
\[
 A_I=\operatorname{coker}
 (B_I^{\mathsf t}:\Z^{[d]\setminus I}\longrightarrow\Z^I).
\]
For a fixed independent set $J$, choose a point of $T_J$ to
translate all $T_I$, $I\subseteq J$, to character kernels.
Set $L_I=\operatorname{coker}(B_I)$. The geometric components are
$\Hom(\operatorname{tors}L_I,\Q/\Z)$, after compatible choices of
roots of unity; dualizing the free presentation of $L_I$ identifies
this group with $\Ext^1(L_I,\Z)=A_I$. Naturality identifies the map
from $A_J$ to $A_I$ with coordinate projection. Since the restored rows
are zero on the columns of $I$, this projection is well defined
and surjective. Consequently, $Q_J$ is the quotient of $\Z^{A_J}$ by the
indicators of fibers of the projections $A_J\to A_{J\setminus\{j\}}$.
Translation choices only permute the component labels.

For example, three isolated vertices with frozen block
$\left(\begin{smallmatrix}1&1&1\\0&2&0\\0&0&2\end{smallmatrix}\right)$
give $\CH^1=0$, $\CH^2=\Z^3$, and $\CH^3=\Z/2$ in
characteristic different from two. The last stratum has four
rational components. The fibers of the three maps
$\pi_0(T_{\{1,2,3\}})\to\pi_0(T_{\{i,j\}})$, $1\le i<j\le3$,
are precisely the six two-element subsets of $\pi_0(T_{\{1,2,3\}})$. The indicators of these subsets generate the even-sum
sublattice of $\Z^4$, which proves the last assertion. This example shows that the integral
component formula can retain torsion lost by rationalization.

\subsection{Vector bundles for connected strata}
\begin{corollary}\label{cor:component-bundle-range}
Let $k$ be perfect of characteristic different from two, and let $X$
be an acyclic cluster variety satisfying \AR. Suppose every $T_I$ is
connected over $k$. Every vector bundle on $X$ of rank greater than
$\alpha(\Gamma)$ is free. If $\alpha(\Gamma)\le2$, the Euler class gives
\[
\operatorname{Vect}^{\mathrm{or}}_2(X)\simeq\CW^2(X),
\qquad
\operatorname{Vect}_2(X)\simeq
\CW^2(X)/\Gamma(X,\mathcal O_X)^\times,
\]
and every rank-two vector bundle becomes free after taking the direct
sum with a trivial line bundle.
If $\alpha(\Gamma)\le1$, all vector bundles are free.
\end{corollary}
\begin{proof}
Theorem~\ref{cc:cor:fixed} gives $K_0(X)=\Z$, and
Proposition~\ref{prop:AR-coefficients} gives the required bound on cohomological
dimension. The assertions now follow from Proposition~\ref{prop:bundle-range}.
\end{proof}
Under \AS, the integer Witt complex of Theorem~\ref{thm:universal-witt}
computes the Chow--Witt group $\CW^2(X)$ in this statement. For connected nonsplit
strata under \AR, its terms instead involve coefficients over the
finite separable fields occurring in $E_I$; one cannot replace them
by the constant tensor factor $W(k)$ used in the split case.

\section{The integer Witt complex and Chow--Witt groups}\label{sec:witt}\label{subsec:universal-witt}
The Witt specialization of the cellular differential is defined over $\Z$.
We will use this to compute the positive Chow--Witt groups, including their
possible torsion, by the same finite complex over every perfect field of
characteristic different from two.
Throughout this section, $\Btil$ is an acyclic seed satisfying \AS, with skew-symmetrizable principal part, and $X_k$ denotes its cluster variety over such a field $k$.

We begin with $V=\F_2^d$, write $z^a$ for the basis element of $\Z[V]$ indexed by $a\in V$, and denote the coordinate vectors and exchange columns modulo two by $e_i$ and $b_i$.
For $I\in\Ind(\Gamma)$, we define
\begin{equation}\label{eq:universal-stratum}
R_I=\Z[V]/(z^{e_i}-1,\ z^{b_i}+1:i\in I),
\qquad
C^p_{\Witt}(\Btil;\Z)=\bigoplus_{\substack{I\in\Ind(\Gamma)\\|I|=p}}R_I.
\end{equation}
The sign in $z^{b_i}+1$ records the equation $\chi_i=-1$ on the stratum and is part of the integral calculation.
If $J=I\cup\{j\}$ is independent, let $q_{I,J}:R_I\to R_J$ be the quotient map and let $\nu_J(j)$ be the number of elements of $J$ smaller than $j$.
The $j$th coordinate $a_j\in\{0,1\}$ of $a\in V$ defines the face map
\begin{equation}\label{eq:universal-residue}
\delta_{I,J}(z^a)=(-1)^{\nu_J(j)}a_jq_{I,J}(z^a).
\end{equation}
We take the sum of these face maps as the differential of $C^\bullet_{\Witt}(\Btil;\Z)$.

\begin{theorem}\label{thm:universal-witt}
The groups in \eqref{eq:universal-stratum} form a finite complex of finite free abelian groups, with
\[
\rk R_I=2^{d-2|I|}.
\]
For every $q\geq0$ there are isomorphisms, compatible with extensions of the base field,
\begin{equation}\label{eq:universal-witt-cohomology}
H^q_{\Nis}(X_k,\Wsh)
\simeq H^q\bigl(C^\bullet_{\Witt}(\Btil;\Z)\otimes_{\Z}\Witt(k)\bigr).
\end{equation}
For every $p>0$, the change-of-coefficient maps induce isomorphisms
\begin{equation}\label{eq:universal-CW}
\CW^p(X_k)\xrightarrow{\sim}H^p_{\Nis}(X_k,\Ish^p)
\xrightarrow{\sim}H^p_{\Nis}(X_k,\Wsh).
\end{equation}
In particular, the same integer complex computes all positive Chow--Witt groups, and depends only on $\Ind(\Gamma)$ and $\Btil\bmod2$.
\end{theorem}
\begin{proof}
We first check that the stratum groups are free. The vectors $e_i,b_i$, $i\in I$, are linearly independent in $V$.
Indeed, projection to the coordinates indexed by $I$ kills the $b_i$ by independence of $I$, while detecting the coefficients of the $e_i$. The remaining coefficients vanish because \AS\ makes the selected columns independent modulo two.
Their span therefore admits the sign character $e_i\mapsto1$, $b_i\mapsto-1$.
A complementary subspace of $V$ identifies $R_I$ with the integral group algebra of an elementary abelian group of rank $d-2|I|$, proving freeness and the rank formula.

For $j\notin I$ with $I\cup\{j\}$ independent, the $j$th coordinate vanishes on every old relation vector $e_i,b_i$.
It follows that \eqref{eq:universal-residue} respects the defining signed relations of $R_I$.
For two distinct vertices added to $I$, the coordinate projections commute, as do the quotient maps, while the two incidence signs are opposite.
Consequently $\delta^2=0$.

We now compare this complex with Witt cohomology. Successive evaluation at $1$ and contraction on a split torus $T$ give
\[
\Wsh(T)=\Witt(k)[M/2M],
\]
where $M$ is its character lattice and the character $t^a$ corresponds to $\la t^a\ra$.
This is the torus decomposition underlying \cite[Section 3.3]{MS}.
On the stratum indexed by $I$, the relations on these classes are precisely those in \eqref{eq:universal-stratum}: normal coordinates are omitted and $\la\chi_i\ra=\la-1\ra=-1$.
Hence $R_I\otimes\Witt(k)\simeq\Wsh(T_I)$.
The second Witt residue of $\la x_j^a u\ra$, with $u$ unramified along $x_j=0$, is zero when $a$ is even and $\la u\ra$ when $a$ is odd.
With the ordered normal frame, the full cellular differential is therefore \eqref{eq:universal-residue} tensored with $\Witt(k)$.
The cellular cohomology theorem gives \eqref{eq:universal-witt-cohomology}, compatibly with field extension.

To compare the coefficient sheaves, we choose coordinates $t_1,\ldots,t_s$ on a stratum torus and set $\pf{t_A}=\prod_{a\in A}(1-\la t_a\ra)$.
The triangular change from character classes to Pfister monomials gives
\begin{align*}
\Wsh(T)&=\Witt(k)\oplus\bigoplus_{\varnothing\ne A\subseteq[s]}\Witt(k)\pf{t_A},\\
\Ish(T)&=I(k)\oplus\bigoplus_{\varnothing\ne A\subseteq[s]}\Witt(k)\pf{t_A}.
\end{align*}
The cellular complexes for $\Ish^p$ and $\Wsh$ coincide from degree $p$ onward.
In degree $p-1$, the only difference in each stratum term is the inclusion
$I(k)\subseteq\Witt(k)$ in the constant coefficient group.
Constant Witt classes have zero residue, so the two complexes have the same image in degree $p$.
This proves the second isomorphism in \eqref{eq:universal-CW}.
The first follows from the Milnor--Witt coefficient square and \cite[Proposition 2.11]{HW}: when the Chow groups have no $2$-torsion, the comparison identifies $\CW^p(X)$ with the fiber product of $H^p(X,\Ish^p)$ and a subgroup of $\CH^p(X)$ over $\CH^p(X)/2$. Positive Chow vanishing from Corollary~\ref{cc:cor:universal} reduces this fiber product to $H^p(X,\Ish^p)$.
\end{proof}

\begin{corollary}\label{cor:universal-witt-UCT}
Let $A_p=H^p(C^\bullet_{\Witt}(\Btil;\Z))$.
For $p>0$ one has $A_p\simeq\CW^p(X_{\R})$ and a natural short exact sequence
\begin{equation}\label{eq:universal-witt-UCT}
0\longrightarrow A_p\otimes_{\Z}\Witt(k)
\longrightarrow\CW^p(X_k)
\longrightarrow\operatorname{Tor}^{\Z}_1(A_{p+1},\Witt(k))
\longrightarrow0.
\end{equation}
It splits as a sequence of abelian groups, without a preferred splitting.
If Smith normal form gives
\[
A_p=\Z^{r_p}\oplus\bigoplus_i\Z/a_i,
\qquad
(A_{p+1})_{\mathrm{tors}}=\bigoplus_j\Z/b_j,
\]
then, noncanonically,
\begin{equation}\label{eq:universal-witt-Smith}
\CW^p(X_k)\simeq\Witt(k)^{r_p}
\oplus\bigoplus_i\Witt(k)/a_i\Witt(k)
\oplus\bigoplus_j\Witt(k)[b_j],
\end{equation}
where $\Witt(k)[b]=\ker(b:\Witt(k)\to\Witt(k))$.
\end{corollary}
\begin{proof}
The universal coefficient theorem applies to the finite free complex of Theorem~\ref{thm:universal-witt}. Together with $\Witt(\R)=\Z$, it gives the asserted exact sequence.
Smith normal form then gives \eqref{eq:universal-witt-Smith}.
Thus the torsion of the next cohomology group contributes in each intermediate degree; in the highest degree, this contribution is absent.
\end{proof}

To pass from oriented bundles to bundles, we also need the unit action, which is computed by the same matrices.
Every unit on $X_k$ restricts to a Laurent monomial on the initial torus.
For each mutable $i$, its valuation on the open stratum $x_i=0$ forces the exponent of $x_i$ to vanish, since $x_i$ has order one there and all other initial coordinates remain units.
Thus every global unit is a constant times a frozen Laurent monomial.
For $u=\lambda x^f$, multiplication by $\la u\ra$ is represented on $R_I\otimes\Witt(k)$ by multiplication by $z^f$ and the scalar $\la\lambda\ra$.
These maps commute with the differential, because frozen monomials are unramified in every normal direction.
Their induced cohomology action can therefore be computed before choosing the noncanonical splitting in \eqref{eq:universal-witt-Smith}.

\begin{remark}\label{rem:universal-I-powers}
The construction also computes every $\Ish^j$-coefficient complex.
In a Pfister basis, this amounts to replacing the $A$th coefficient of its degree-$p$ stratum summand by $I^{\max(j-p-|A|,0)}(k)\subseteq\Witt(k)$ and restricting the same integer differential.
This follows from the evaluation--contraction decomposition of $\Ish^{j-p}(T_I)$.
In particular $\Ish^j\to\Wsh$ induces an isomorphism on $H^q$ whenever $q\geq j$ and $q>0$.
\end{remark}

\section{Even exchange matrices and independence complexes}\label{sec:even-graphs}

We now give a description of the Witt complex in terms of the mutable graph
for a class of coefficient patterns that includes principal coefficients
and many seeds that are not RFR. This description computes the Chow--Witt
groups together with the unit action needed to forget orientations.
Under a torsion-freeness condition on independence complexes, we also
reduce the full cellular complex and determine the
cohomological dimension for all strictly invariant coefficient sheaves.

Throughout this appendix $k$ is perfect of characteristic different from
two. Let $B$ be an acyclic skew-symmetric $n\times n$ integer matrix
whose entries are even, and let $\Gamma$ be its mutable support graph.
We choose a partition
\[
[n]=V_1\sqcup\cdots\sqcup V_q
\]
into nonempty cliques, and write $c(i)=a$ for $i\in V_a$. We define
$C_{ai}=1$ if $i\in V_a$ and $C_{ai}=0$ otherwise. We write
$X_{B,c}$ for the cluster variety of $\Btil=\binom B C$, with all
frozen variables $t_1,\ldots,t_q$ inverted, and put $d=n+q$.
Principal coefficients correspond to the partition into singletons.
An independent set meets each $V_a$ at most once, so its selected frozen
rows contain an identity submatrix. Consequently \AS\ holds, even when
$\Btil$ has deficient rank. Theorem~\ref{cc:cor:fixed} and
Corollary~\ref{cc:cor:universal} give
\begin{equation}\label{eq:even-stable}
K_0(X_{B,c})=\Z,\qquad
\CH^p((X_{B,c})_L)=0\quad(p>0)
\end{equation}
for every field extension $L/k$.

For a vertex set $T$, we regard $\Ind(\Gamma[T])$ as a simplicial complex,
including its empty simplex, and use augmented reduced chains. In
particular $\widetilde H_{-1}(\Ind(\Gamma[\varnothing]);\Z)=\Z$.
We impose the torsion-freeness condition \eqref{eq:graph-TF}
only when stated.
We use the notation
\begin{align}
\beta_j(T)&=\rk_{\Z}\widetilde H_j(\Ind(\Gamma[T]);\Z),\nonumber\\
\lambda(\Gamma)&=\max\{p\ge0:
\widetilde H_{p-1}(\Ind(\Gamma[T]);\Z)\ne0
\text{ for some }T\},\label{eq:graph-lambda}\\
w_c(T)&=\prod_{a\notin c(T)}(1+2^{|V_a|}),\qquad
N_p(\Gamma,c)=\sum_{T\subseteq[n]}w_c(T)\beta_{p-1}(T).
\label{eq:graph-multiplicity}
\end{align}
The empty set makes the maximum well defined. The torsion-freeness
condition is not needed for the following Chow--Witt calculation.

\subsection{Chow--Witt groups and changes of orientation}
For $S\subseteq[n]$ and $F\subseteq[q]$, we put
$T(S,F)=S\setminus c^{-1}(F)$ and define, for $p>0$,
\begin{equation}\label{eq:graph-uncompressed}
\mathcal M_p(\Gamma,c;k)=
\bigoplus_{S\subseteq[n],\,F\subseteq[q]}
\widetilde H^{p-1}(\Ind(\Gamma[T(S,F)]);\Witt(k)).
\end{equation}
The superscript denotes ordinary simplicial cohomology with constant
coefficients in the abelian group $\Witt(k)$; it retains the coefficient
effects of integral torsion in the independence complex.

\begin{theorem}\label{thm:even-graph-CW}
For every $p>0$ there are $\Witt(k)$-module isomorphisms
\begin{align}
\CW^p(X_{B,c})&\simeq\mathcal M_p(\Gamma,c;k)\nonumber\\
&\simeq\bigoplus_{T\subseteq[n]}
\widetilde H^{p-1}(\Ind(\Gamma[T]);\Witt(k))^{w_c(T)}.
\label{eq:graph-CW}
\end{align}
They are compatible with field extension, using the indicated coordinate
and normal-frame choices. If \eqref{eq:graph-TF} holds, choosing integral
homology bases gives a noncanonical isomorphism
\[
\CW^p(X_{B,c})\simeq\Witt(k)^{N_p(\Gamma,c)}.
\]
For principal coefficients, $w_c(T)=3^{n-|T|}$.
\end{theorem}
\begin{proof}
We compute the differential in coordinates adapted to the coefficient partition.
Put $z_a=-t_a$. On the stratum with zero set $I$, the free torus
coordinates are $x_j$, $j\notin I$, and $z_a$, $a\notin c(I)$.
For a face $J=I\cup\{i\}$ the new equation is
\begin{equation}\label{eq:even-face-equation}
z_{c(i)}=\prod_{j\notin J}x_j^{-B_{ji}}.
\end{equation}
Since every exponent is even, the right-hand side has
one-dimensional class $1$ in the Witt ring. We write $e_u=\la u\ra-1$; these elements
commute and satisfy $e_u^2=-2e_u$. By the torus decomposition in
Theorem~\ref{thm:universal-witt}, the monomials
\[
\prod_{j\in A_x}e_{x_j}\prod_{a\in F}e_{z_a},\qquad
A_x\subseteq[n]\setminus I,\quad F\subseteq[q]\setminus c(I),
\]
form a $\Witt(k)$-basis of the stratum term. Its differential contracts
$e_{x_i}$ with coefficient one and the normal incidence sign
$(-1)^{\#\{j\in I:j<i\}}$. A monomial containing $e_{z_{c(i)}}$
restricts to zero, while all other symbols remain unchanged.

Consequently the cochain complex splits by $S=I\cup A_x$ and $F$.
Within this block the possible $I$ are exactly the independent subsets
of $T(S,F)$. A simplex with $|I|=p$ lies in cellular degree $p$, so the
block is the augmented simplicial cochain complex of
$\Ind(\Gamma[T(S,F)])$, shifted up by one. Theorem~\ref{thm:universal-witt}
identifies its positive-degree cohomology with $\CW^p$.
The constant summand of weight zero contributes only in degree zero.

For a fixed $T$, the allowable $F$ are subsets of $[q]\setminus c(T)$,
and then $S=T\cup A$ with arbitrary $A\subseteq c^{-1}(F)$. There are
therefore
\[
\sum_{F\subseteq[q]\setminus c(T)}2^{|c^{-1}(F)|}
=\prod_{a\notin c(T)}(1+2^{|V_a|})
\]
such blocks. Under \eqref{eq:graph-TF}, integral unit cancellation reduces
each finite free simplicial complex to its homology. Tensoring
this reduction with $\Witt(k)$ proves the last assertion. When torsion
is present, the uncompressed formula retains the universal-coefficient contribution.
\end{proof}

The uncompressed decomposition also describes the unit action, which
is not determined by the multiplicities $N_p$ alone. For $a\notin F$, we denote by
\[
r_a:\widetilde H^{p-1}(\Ind(\Gamma[T]);\Witt(k))
\longrightarrow
\widetilde H^{p-1}(\Ind(\Gamma[T\setminus V_a]);\Witt(k))
\]
the restriction induced by inclusion of complexes, where $T=T(S,F)$.

\begin{proposition}\label{prop:even-unit-action}
On $\mathcal M_p(\Gamma,c;k)$, multiplication by $\la z_a\ra$ is the
commuting involution $A_a$ defined on the $(S,F)$ summand by
\begin{equation}\label{eq:graph-unit-action}
A_a(v)=
\begin{cases}
-v,&a\in F,\\
v+r_a(v),&a\notin F.
\end{cases}
\end{equation}
In the second line the two terms belong to the $(S,F)$ and
$(S,F\cup\{a\})$ summands, respectively. Moreover
\[
\mathcal O(X_{B,c})^\times
=k^\times\times\prod_{a=1}^q t_a^{\Z}.
\]
Thus the full unit action is generated by the $A_a$ and scalar
multiplication by $\la b\ra$ for $b\in k^\times$.
\end{proposition}
\begin{proof}
If the symbol $e_{z_a}$ already occurs, multiplication by
$1+e_{z_a}$ is $-1$ on that monomial. Otherwise the additional term inserts
$e_{z_a}$. This insertion vanishes on strata whose independent zero
set meets $V_a$, by \eqref{eq:even-face-equation}, and preserves the mutable
symbols on every other stratum. It therefore agrees with the simplicial restriction
avoiding $V_a$. This proves \eqref{eq:graph-unit-action} on cochains.
The same description proves that the $A_a$ commute and square to one.

It remains to identify the global units. A global unit restricts to a Laurent unit on the initial torus, and hence has the form
$b x_1^{a_1}\cdots x_n^{a_n}t^f$. The singleton zero stratum of $x_i$
is nonempty and irreducible. Along its closure, $x_i$ has valuation
one, and all other initial coordinates have valuation zero. Since a unit
has valuation zero, every $a_i=0$. Conversely, constants and
frozen monomials are units. Since $\la t_a\ra=-\la z_a\ra$ on Witt
coefficients and the constant $-1$ already acts by $-1$, the stated
generators give precisely the full action.
\end{proof}

\subsection{The full differential and its cohomological dimension}
To bound cohomology for all strictly invariant sheaves, we return to the
full cellular differential and retain multiplication by the hyperbolic
form, which vanishes in the preceding Witt calculation. We put
\[
g=\la-1\ra,\qquad h=1+g,\qquad
R_h=\Z[h]/(h^2-2h).
\]
For $s\ge1$, let $\mathscr E_s^p$ be the free $R_h$-module on triples
$(I,A_x,F)$ such that
\[
I\in\Ind(\Gamma),\quad |I|=p,\quad
A_x\subseteq[n]\setminus I,\quad F\subseteq[q]\setminus c(I),
\quad p+|A_x|+|F|=s.
\]
Introduce formal exterior symbols $e_1,\ldots,e_n$ and
$f_1,\ldots,f_q$, distinct from the Witt classes $e_u=\la u\ra-1$
used above. For fixed $I$, we represent $(I,A_x,F)$ by $e_{A_x}f_F$,
with increasing indices in each group and all mutable symbols before
the frozen symbols.
For $J=I\cup\{i\}$ independent, the face map is left contraction in
$e_i$, followed by the substitution
\begin{equation}\label{eq:even-formal-face}
f_{c(i)}\longmapsto-h\sum_{j\notin J}(B_{ji}/2)e_j.
\end{equation}
If the contracted monomial lacks $f_{c(i)}$, the substitution leaves it unchanged.
All other symbols map identically, and repeated exterior
symbols vanish. We denote the resulting cochain complex by
$\mathscr E_s^\bullet$, and its transpose chain complex by
$\mathscr D_s$, and write $D_s=\mathscr D_s\otimes_{R_h,h\mapsto2}\Z$.
The integer matrices of $D_s$ are used in Section~\ref{app:even-scalar}.

\begin{theorem}\label{thm:even-model}
The differential on $\mathscr E_s^\bullet$ defined by these face maps squares to zero over $R_h$.
After homogeneous changes of basis, the full cellular complex is
\begin{equation}\label{eq:even-Rh-model}
C_*^\Sigma(X_{B,c})\cong
\Z[0]\oplus\bigoplus_{s=1}^{d}
\mathscr D_s\otimes_{R_h}\KMW_s,
\end{equation}
where $h$ acts on $\KMW_s$ by the hyperbolic form. The tensor notation
means replacing each free $R_h$ basis term by $\KMW_s$ and evaluating
the matrix entries in this action; it preserves chain homotopies.
\end{theorem}
\begin{proof}
The Milnor--Witt identities, with $\rho=[-1]$ and $\epsilon=-g$, give
\[
\eta h=0,\qquad \rho h=0,\qquad gh=h,\qquad
\epsilon h=-h,\qquad [u^{2a}]=ah[u]\quad(a\in\Z).
\]
These follow from the standard symbol relations
\cite[Section 1]{Fasel}. Thus \eqref{eq:even-face-equation} gives
\[
[z_{c(i)}]=-h\sum_{j\notin J}(B_{ji}/2)[x_j].
\]
Every nonlinear product correction contains $\eta h$ and therefore vanishes.
Likewise, a repeated symbol satisfies $[x]^2=\rho[x]$, so it vanishes in a term
multiplied by $h$. Substitution therefore has precisely the exterior
form \eqref{eq:even-formal-face}.

We next normalize the pure contractions. Put $S=I\cup A_x$ and
\[
\mu(I,A_x)=\sum_{j\in I}\#\{v\in S:v<j\}\pmod2.
\]
We multiply the corresponding homogeneous generator by $g^{\mu(I,A_x)}$.
For a face adding $i$, we put $r=\#\{j\in A_x:j<i\}$ and
$\nu=\#\{j\in I:j<i\}$. The original coefficient is
$g^\nu\epsilon^r=(-1)^rg^{\nu+r}$, and the difference of the two
$\mu$ values is $\nu+r$. The new pure coefficient is therefore
$(-1)^r$, the ordinary left-contraction sign. Terms already divisible
by $h$ are unchanged by the $g$ factors, and their $\epsilon$
commutations give ordinary exterior signs because $\epsilon h=-h$.
This change of basis preserves $s=p+|A_x|+|F|$.

For two successively added independent vertices $i,j$, their colors
are distinct and $B_{ij}=B_{ji}=0$. The substitutions commute with
one another and with the other normal contraction; the two exterior
contractions anticommute. Hence the square-zero identity holds even
over $\Z[H]$ with $H$ in place of $h$. Applying the torus
decomposition to Theorem~\ref{thm:full-differential} now gives
\eqref{eq:even-Rh-model}; all nonconstant summands are copies of the
indicated projective $\KMW_s$.
\end{proof}

\begin{lemma}\label{lem:even-reduction}
Assume \eqref{eq:graph-TF}. For each $s$, there is a chain-homotopy
equivalence over $R_h$ from $\mathscr D_s$ to a finite free complex
whose differential is $hA_s$, with $A_s$ an integer chain complex.
Its degree-$p$ generators form a basis of
\[
\bigoplus_{\substack{S\subseteq[n],\,F\subseteq[q]\\|S|+|F|=s}}
\widetilde H_{p-1}(\Ind(\Gamma[T(S,F)]);\Z).
\]
In particular no generator has degree above $\lambda(\Gamma)$.
Every degree-$p$ generator has weight $s\ge2p$.
\end{lemma}
\begin{proof}
We write the differential as $U+hV$. The pure part preserves $(S,F)$.
Within each such block the additional sign change $(-1)^{\mu(I,A_x)}$
changes its left-contraction signs to simplicial incidence signs.
After transposition this is the augmented simplicial chain complex,
with simplex cardinality as cellular degree. By \eqref{eq:graph-TF}
it splits over $\Z$ into unit contractible pairs and its free homology.
We choose this splitting separately for each $(S,F)$, so that all chosen
contractions preserve $F$.

In cochains $V$ removes exactly one frozen symbol; in chains it adds
one. It therefore strictly increases a finite filtration by frozen
cardinality, while $U$ and its chosen contraction preserve it.
We now perform the chosen integer unit cancellations in $U+hV$. A correction
to a pivot would have to follow a path containing $V$ and return to
the same frozen cardinality; no such path exists. Thus every pivot remains
exactly $1$ or $-1$, and the finite cancellation leaves precisely the
free $U$-homology generators. Each residual boundary is a sum of paths
containing at least one $hV$, so its entries have zero constant term
in $h$. Since $h^\ell=2^{\ell-1}h$ for $\ell\ge1$, the remaining
differential is $hA_s$ with integer entries. Its square is
$2hA_s^2=0$. The additive group of $R_h$ is free on $1,h$, so
$A_s^2=0$ integrally. The construction gives chain-homotopy equivalences before
evaluation in any field.

For the weight bound, a flag complex with nonzero reduced homology
in degree $p-1$ over a field has at least $2p$ vertices.
For $p=1$ this follows from disconnectedness. For $p>1$, choose
an induced subcomplex with minimal vertex set supporting nonzero
homology. Deleting any vertex kills that homology, so the
Mayer--Vietoris sequence gives nonzero reduced homology in degree
$p-2$ in its link. The link is flag and has at least $2p-2$
vertices by induction. There is also a nonneighbor of the chosen
vertex, since otherwise the complex would be a cone. This proves
the bound. Independence complexes are flag, and under
\eqref{eq:graph-TF} their nonzero integral homology remains
nonzero over $\Q$. Thus a degree-$p$ generator satisfies
$2p\le|T(S,F)|\le|S|+|F|=s$.
\end{proof}

\begin{theorem}\label{thm:even-cd}
Under \eqref{eq:graph-TF},
\[
\operatorname{cd}_{\A^1}(X_{B,c})=\lambda(\Gamma).
\]
\end{theorem}
\begin{proof}
Theorem~\ref{thm:even-model} and Lemma~\ref{lem:even-reduction} give
a projective cellular model concentrated in degrees zero through
$\lambda(\Gamma)$. Applying $\Hom_{\calA}(-,M)$ for every strictly
invariant $M$ proves the upper bound. If $\lambda(\Gamma)>0$, some
induced independence complex has nonzero free homology in degree
$\lambda(\Gamma)-1$. Every such subset occurs in
\eqref{eq:graph-CW}, so $H^{\lambda(\Gamma)}(X_{B,c},\Wsh)$ has a
nonzero summand $\Witt(k)$. This group is nonzero over every field of
characteristic different from two, since $\la1\ra$ has odd dimension
modulo two. This proves the lower bound. The case $\lambda=0$
follows from the upper bound and $H^0(X_{B,c},\Z)\ne0$.
\end{proof}

Condition~\eqref{eq:graph-TF} permits integral unit cancellation
of the pure part. Under this condition, Section~\ref{app:even-scalar}
gives every cellular homology sheaf as a sum of Milnor--Witt sheaves
and kernels and cokernels of scalar hyperbolic multiplication.
For induced independence complexes with torsion, the pure part has
nonunit Smith blocks, which enter the general model
\eqref{eq:even-Rh-model}.

\subsection{Graph criteria for the rank-two range}
We now give graph-theoretic sufficient conditions for
\eqref{eq:graph-TF} and $\lambda(\Gamma)\le2$.

\begin{lemma}\label{lem:graph-cover}
Let $C_0$ be a vertex cover of a finite graph $\Gamma$, and put
$a=\alpha(\Gamma[C_0])$. For every $T$, the augmented integral
chains of $\Ind(\Gamma[T])$ are chain-homotopy equivalent to a
finite free complex concentrated in degrees $-1,\ldots,a-1$.
If $a\le2$, condition \eqref{eq:graph-TF} holds and
$\lambda(\Gamma)\le2$.
\end{lemma}
\begin{proof}
We replace the cover by $C_0\cap T$ and let $S_0=T\setminus C_0$,
which is independent. A subset of $C_0$ dominates $S_0$ if every vertex
of $S_0$ has a neighbor in that subset. Consider the quotient of the
augmented chains by all simplices meeting $S_0$ and all simplices in
$C_0$ that do not dominate $S_0$.
These simplices form a subcomplex: deleting a cover vertex cannot turn a
nondominating subset into a dominating one, and deleting a sole
$S_0$ vertex leaves a nondominating cover subset. The quotient has
as basis the independent subsets of $C_0$ that dominate $S_0$;
its boundary deletes a vertex only when domination persists.

To compare this quotient with the original complex, we filter by the number of cover vertices. For a
fixed independent cover subset $I$, the remaining direction is the
augmented simplex on $S_0\setminus N(I)$. It is contractible unless
this set is empty, in which case it consists of one $\Z$ in degree
$|I|-1$. It follows that the quotient map is a quasi-isomorphism. Its kernel is
a bounded acyclic complex of free abelian groups, and is therefore contractible.
The quotient has maximum degree $a-1$. If $a\le2$, its top possible
homology $H_1$ is free, and reduced $H_0$ and $H_{-1}$ are always
free. All higher groups vanish, proving the last assertion.
\end{proof}

\begin{corollary}\label{cor:graph-chordal}
If $\Gamma$ is chordal, condition \eqref{eq:graph-TF} holds and
\[
\lambda(\Gamma)=\operatorname{indmatch}(\Gamma),
\]
the largest size of an induced matching. If the complement of
$\Gamma$ is chordal, condition \eqref{eq:graph-TF} holds and
$\lambda(\Gamma)\le1$. Conversely, $\lambda(\Gamma)\le1$
implies that its complement is chordal.
\end{corollary}
\begin{proof}
The independence complex of a nonempty chordal graph is contractible or
homotopy equivalent to a finite wedge of spheres \cite{Kawamura}.
Every nonempty induced subgraph is again chordal; the empty induced
subgraph contributes only $\widetilde H_{-1}=\Z$ under our convention.
Thus \eqref{eq:graph-TF} holds. Hochster's formula identifies
the maximum in \eqref{eq:graph-lambda} with the regularity of the
independence Stanley--Reisner ring over any field in this torsion-free
case. For chordal graphs that regularity is the induced matching
number \cite[Corollary 6.9]{HVT}; see also \cite[Section 3 and Theorem 13]{WoodroofeMatchings}.

If the complement is chordal, perfect elimination shows that each
connected component of each induced clique complex is contractible.
These clique complexes are exactly the independence complexes of
induced subgraphs of $\Gamma$. Consequently their reduced homology is free and
concentrated in degree zero, apart from the empty complex. Conversely,
a nonchordal complement contains an induced cycle of length at least
four. Its clique complex is a circle, giving $\lambda(\Gamma)\ge2$.
\end{proof}

For example, $\Gamma=K_{a,b}$ with $a,b\ge1$ has
$\lambda(\Gamma)=1$: its independence complex is the disjoint union
of two simplices, and every induced subgraph has the same form or
is edgeless. This can be much smaller than
$\min_{C_0\text{ vertex cover}}\alpha(\Gamma[C_0])=\min(a,b)$.

\subsection{Vector bundles determined by independence complexes}\label{subsec:graph-bundles}
We combine the cellular reduction of
Lemma~\ref{lem:even-reduction} with stable freeness to classify
vector bundles of rank
$\lambda(\Gamma)$.

\begin{theorem}\label{thm:graph-bundles}
Suppose \eqref{eq:graph-TF} holds, and put $r=\lambda(\Gamma)$.
Every vector bundle on $X_{B,c}$ of rank greater than $r$ is free.
If $r\le1$, all vector bundles are free. If $r\ge3$ is odd,
every rank-$r$ vector bundle is also free.

If $r\ge2$ is even, the Euler class gives bijections
\begin{align}
\operatorname{Vect}^{\mathrm{or}}_r(X_{B,c})
&\simeq\CW^r(X_{B,c})\simeq\mathcal M_r(\Gamma,c;k)
\simeq\Witt(k)^{N_r(\Gamma,c)},\label{eq:graph-oriented-bundles}\\
\operatorname{Vect}_r(X_{B,c})
&\simeq\mathcal M_r(\Gamma,c;k)/
\langle A_1,\ldots,A_q,\ \la b\ra\ (b\in k^\times)\rangle.
\label{eq:graph-unoriented-bundles}
\end{align}
Here $A_a$ is the involution defined in
\eqref{eq:graph-unit-action}, and the last quotient is an orbit set.
Every rank-$r$ vector bundle $E$ then satisfies
$E\oplus\mathcal O\cong\mathcal O^{r+1}$, and nonfree rank-$r$
vector bundles exist over every such field.
\end{theorem}
\begin{proof}
By \eqref{eq:even-stable} and Theorem~\ref{thm:even-cd}, the hypotheses of
Proposition~\ref{prop:bundle-range} hold with $h=r$. This proves
freeness above rank $r$, and in every rank if $r\le1$.
We write $X=X_{B,c}$ for the remaining argument.

Suppose first that $r\ge3$ is odd. Put
$\Pi_r=\pi_r^{\A^1}(BSL_r)$. By
\cite[Corollary~3.11 and Theorem~3.14]{AFSpheres}, there are
an exact sequence and an epimorphism
\[
0\longrightarrow\mathbf S_{r+1}\longrightarrow\Pi_r
\longrightarrow\mathbf K_r^Q\longrightarrow0,
\qquad \KM_{r+1}/r!\twoheadrightarrow\mathbf S_{r+1}.
\]
Exactness of contraction and vanishing of negative Milnor and
Quillen $K$-theory imply $(\Pi_r)_{-s}=0$ for $s>r+1$.
By Lemma~\ref{lem:even-reduction}, the degree-$r$ cellular terms
have weights $s\ge2r>r+1$, so $H^r(X,\Pi_r)=0$.
Stabilization induces an equivalence on the $(r-1)$st Postnikov
truncations of $BSL_r$ and $BSL_\infty$
\cite[Theorem~3.2]{AFSpheres}. Stable freeness therefore makes
the class of an oriented bundle trivial at this stage. Its lifts
to the next stage form a quotient of $H^r(X,\Pi_r)$, hence a
singleton. Higher obstructions and lifting ambiguities vanish
because $\operatorname{cd}_{\A^1}X=r$; convergence is as in
Proposition~\ref{prop:bundle-range}. Thus the bundle is free.

Now let $r\ge2$ be even. Since every rank-$(r+1)$ bundle is free,
the fiber sequence
$\A^{r+1}\setminus0\to BSL_r\to BSL_{r+1}$ gives a surjection
\[
q:[X,\A^{r+1}\setminus0]_{\A^1}
\longrightarrow\operatorname{Vect}^{\mathrm{or}}_r(X).
\]
A free trivialization can be adjusted by a diagonal unit to
respect the determinant orientation. The first Postnikov map
identifies the source with $H^r(X,\KMW_{r+1})$, since the fiber
is $(r-1)$-connected and $\operatorname{cd}_{\A^1}X=r$.
The reduced complex of Lemma~\ref{lem:even-reduction} preserves
weight, and only weights $s\ge2r$ contribute to degree-$r$
cohomology. On each such weight complex, multiplication by
$\eta$ induces the coefficient isomorphism
\[
K_{r+1-s}^{MW}(k)\xrightarrow{\ \eta\ }K_{r-s}^{MW}(k),
\]
since both indices are negative. It follows that
$H^r(X,\KMW_{r+1})\xrightarrow{\eta}\CW^r(X)$ is an isomorphism.
The universal Euler-class calculation
\cite[Lemma~3.5]{AFSpheres}, together with
\cite[Theorem~5.3.2]{AFEuler}, identifies $e\circ q$ with this
map up to a fixed unit in $\GW(k)$. Thus $e\circ q$ is bijective.
Since $q$ is surjective, both $q$ and $e$ are bijective.

Theorem~\ref{thm:even-graph-CW} computes $\CW^r(X)$, and
Proposition~\ref{prop:even-unit-action} gives the unit action for
forgetting orientations. Since $r=\lambda(\Gamma)>0$, we have
$N_r>0$ and $\Witt(k)\ne0$. Units act invertibly and fix zero,
so nonzero Euler classes give nonfree bundles after forgetting
orientation. Finally $E\oplus\mathcal O$ is free because its
rank is greater than $r$.
\end{proof}

Theorem~\ref{thm:graph-bundles} and Corollary~\ref{cor:graph-chordal}
in particular show that every vector bundle is free for an even seed of this
coefficient form whose complementary graph is chordal. For chordal
$\Gamma$, this sufficient freeness bound is expressed by its induced
matching number.

The punctured-space calculations of Asok--Bachmann--Hopkins and
Liu--Yang extend odd-rank freeness and even-rank Euler classification
below $\lambda(\Gamma)$ in characteristic zero.

\begin{corollary}\label{cor:graph-later-ranks}
Assume \eqref{eq:graph-TF} and $\operatorname{char}k=0$.
Let $r\ge4$ and $\lambda(\Gamma)\le r+2$.
If $r$ is odd, every rank-$r$ vector bundle on $X_{B,c}$ is free.
If $r$ is even, the Euler class and the unit action give the
classifications \eqref{eq:graph-oriented-bundles} and
\eqref{eq:graph-unoriented-bundles}, with this value of $r$,
and every rank-$r$ bundle $E$ satisfies
$E\oplus\mathcal O\cong\mathcal O^{r+1}$.
Moreover, every rank-three vector bundle is free if
$\lambda(\Gamma)\le4$.
\end{corollary}
\begin{proof}
Write $X=X_{B,c}$ and $F_m=\A^m\setminus0$.
By \cite[Theorem~7.2.1]{ABH} and
\cite[Theorem~4.2.1]{LiuYang}, respectively,
\[
\begin{aligned}
(\pi_m^{\A^1}F_m)_{-s}&=0 &&(m\ge4,\ s>m+2),\\
(\pi_{m+1}^{\A^1}F_m)_{-s}&=0 &&(m\ge5,\ s>m+4).
\end{aligned}
\]
For the first assertion, the contracted Milnor term is zero,
and the Grothendieck--Witt term has negative lower index and
shifted Witt degree $-1$, hence vanishes on fields.
For the second, the $(m+4)$th contraction is $\KM_0/2$, whose
next contraction is zero.

Let $\beta_N:\pi_N^{\A^1}(BSL_N)\to\KMW_N$ be induced by
$SL_N\to F_N$. We claim that
$(\ker\beta_N)_{-s}=0$ for $N\ge4$ and $s>N+2$.
For odd $N$, the source vanishes by the first unstable sheaf
calculation used in Theorem~\ref{thm:graph-bundles}.
For even $N$, the adjacent fiber sequence gives
\[
\pi_{N+1}^{\A^1}(BSL_{N+1})\longrightarrow\KMW_{N+1}
\xrightarrow{\alpha_N}\pi_N^{\A^1}(BSL_N)
\longrightarrow\pi_N^{\A^1}(BSL_{N+1}).
\]
After $s>N+2$ contractions the outer terms vanish, so
$\alpha_N$ is an isomorphism. The composite
$\beta_N\alpha_N$ is multiplication by $\eta$
\cite[Lemma~3.5]{AFSpheres}, which is an isomorphism in negative
Milnor--Witt degrees. This proves the claim.
The long exact sequence of $F_{r+1}\to BSL_r\to BSL_{r+1}$
now gives
\[
\begin{aligned}
(\pi_{r+1}^{\A^1}BSL_r)_{-s}&=0 &&(r\ge3,\ s>r+3),\\
(\pi_{r+2}^{\A^1}BSL_r)_{-s}&=0 &&(r\ge4,\ s>r+5).
\end{aligned}
\]
Indeed, the first assertion uses the bound for $\pi_{r+1}^{\A^1}F_{r+1}$
and the claim for $\beta_{r+1}$. One degree higher, the adjacent
terms are $\pi_{r+2}^{\A^1}F_{r+1}$ and
$\pi_{r+2}^{\A^1}BSL_{r+1}$; the latter is covered by the first assertion.

Lemma~\ref{lem:even-reduction} therefore kills the cohomology
with these coefficients in degrees $r+1$ and $r+2$, respectively.
For odd $r$, the argument of Theorem~\ref{thm:graph-bundles}
also gives $H^r(X,\pi_r^{\A^1}BSL_r)=0$.
Stable freeness makes the lower Postnikov class trivial, and
these vanishings give uniqueness of the subsequent lifts.
The remaining groups vanish by the cohomological-dimension
bound; convergence is as in Proposition~\ref{prop:bundle-range}.
Using only the first bound proves odd-rank freeness for
$r\ge3$ and $\lambda\le r+1$, including the rank-three assertion.
Using both proves it for $r\ge4$ and $\lambda\le r+2$.

For even $r$, the odd-rank case just proved makes every
rank-$(r+1)$ bundle free. The fiber sequence therefore gives
a surjection $q:[X,F_{r+1}]_{\A^1}\twoheadrightarrow
\operatorname{Vect}^{\mathrm{or}}_r(X)$, with orientations
adjusted by diagonal units. The punctured-space bounds and
Lemma~\ref{lem:even-reduction} give
\[
H^{r+1}(X,\pi_{r+1}^{\A^1}F_{r+1})=
H^{r+2}(X,\pi_{r+1}^{\A^1}F_{r+1})=
H^{r+2}(X,\pi_{r+2}^{\A^1}F_{r+1})=0.
\]
These vanishings remove both the obstructions and the ambiguities
above the first Postnikov stage; higher groups vanish by
cohomological dimension. Thus
$[X,F_{r+1}]_{\A^1}\simeq H^r(X,\KMW_{r+1})$.
Only weights $s\ge2r$ contribute in degree $r$, and $\eta$ is
an isomorphism on each entire fixed-weight coefficient complex.
The Euler-class comparison in Theorem~\ref{thm:graph-bundles}
therefore makes $e\circ q$ bijective, hence also $e$.
The Chow--Witt formula and the unit action are those of
Theorem~\ref{thm:even-graph-CW} and
Proposition~\ref{prop:even-unit-action}. Finally,
$E\oplus\mathcal O$ is free by the odd-rank case.
\end{proof}

\begin{corollary}\label{cor:graph-stabilization-realization}
Assume \eqref{eq:graph-TF}, and let $k$ be perfect of
characteristic different from two. If $r\ge2$ and
$\lambda(\Gamma)\le r+1$, every rank-$r$ vector bundle $E$ on
$X_{B,c}$ satisfies $E\oplus\mathcal O\cong\mathcal O^{r+1}$.
For even $r$ in this range, the Euler class gives a surjection
\[
\operatorname{Vect}^{\mathrm{or}}_r(X_{B,c})
\twoheadrightarrow\CW^r(X_{B,c}).
\]
In characteristic zero, the stabilization assertion also holds
for $r\ge3$ with $\lambda(\Gamma)\le r+3$, and for $r=2$ with
$\lambda(\Gamma)\le4$. The Euler-class map is also surjective
for even $r\ge4$ with $\lambda(\Gamma)\le r+3$.
\end{corollary}
\begin{proof}
When $\lambda\le r+1$, apply Theorem~\ref{thm:graph-bundles}
to $E\oplus\mathcal O$. It is free if its rank exceeds $\lambda$
or equals an odd $\lambda$. In the remaining case, its Euler
class vanishes because it has a nowhere-zero section, so the
Euler-class classification makes it free as well.
The same argument with Corollary~\ref{cor:graph-later-ranks}
gives the wider characteristic-zero ranges, including rank two
when $\lambda\le4$.

For the Euler-class assertions, let $r$ be even, write
$X=X_{B,c}$ and $F=\A^{r+1}\setminus0$.
The first Postnikov map gives
\[
\varphi:[X,F]_{\A^1}\longrightarrow H^r(X,\KMW_{r+1}).
\]
Its existence obstructions begin in degree $r+2$, so all vanish
if $\lambda\le r+1$.
In characteristic zero, for $r\ge4$ and $\lambda\le r+3$,
the first two lie in
\[
H^{r+2}(X,\pi_{r+1}^{\A^1}F),\qquad
H^{r+3}(X,\pi_{r+2}^{\A^1}F).
\]
The contraction bounds in the proof of
Corollary~\ref{cor:graph-later-ranks} and
Lemma~\ref{lem:even-reduction} make both groups zero:
their cellular weights are at least $2r+4$ and $2r+6$,
respectively. All further existence obstructions have degree
at least $r+4>\lambda$. Thus $\varphi$ is surjective in both ranges.
Let $q:[X,F]_{\A^1}\to\operatorname{Vect}^{\mathrm{or}}_r(X)$
be the fiber map. As in the proof of Theorem~\ref{thm:graph-bundles},
$e\circ q=u\eta\varphi$ for a fixed $u\in\GW(k)^\times$,
and $\eta:H^r(X,\KMW_{r+1})\to\CW^r(X)$ is an isomorphism.
Thus $e$ is surjective.
\end{proof}

\subsection{Integral cellular homology}\label{app:even-scalar}
Let $k$ be perfect of characteristic different from two. We retain the
even acyclic matrix $B$, the clique partition $c:[n]\to[q]$, and the
formal homogeneous complex of Theorem~\ref{thm:even-model}, and assume
\eqref{eq:graph-TF}. We put $d=n+q$ and
$h=\la1\ra+\la-1\ra\in\GW(k)$.
For $1\le s\le d$, let $D_s$ be the finite free integer chain complex
obtained by transposing the formal weight-$s$ cochain complex and
specializing $h$ to $2$. Thus its cochain face maps are left exterior
contraction in $e_i$, followed by
\[
 f_{c(i)}\longmapsto-\sum_j B_{ji}e_j.
\]
We write its homology in invariant-factor form as
\begin{equation}\label{eq:even-integer-invariants}
 H_p(D_s)=\Z^{\gamma_{p,s}}\oplus
 \bigoplus_{m\ge2}(\Z/m)^{t_{p,s,m}}.
\end{equation}
Only finitely many multiplicities are nonzero. To describe the corresponding cellular homology sheaves, we define strictly
$\A^1$-invariant sheaves
\[
 \mathcal Q_s(a)=\operatorname{coker}(ah:\KMW_s\to\KMW_s),
 \qquad
 \mathcal T_s(a)=\ker(ah:\KMW_s\to\KMW_s)
 \quad(a\ge1).
\]

\begin{theorem}\label{thm:even-scalar}
Every nonzero $t_{p,s,m}$ in \eqref{eq:even-integer-invariants} has
$m$ even. There are noncanonical isomorphisms of cellular homology
sheaves
\begin{equation}\label{eq:even-scalar-homology}
\begin{split}
 H_p(C_*^\Sigma(X_{B,c}))\simeq {}&
 \Z^{\delta_{p0}}\oplus
 \bigoplus_{s=1}^d(\KMW_s)^{\oplus \gamma_{p,s}}\\
 &\oplus\bigoplus_{\substack{1\le s\le d\\m\ge2\text{ even}}}
     \mathcal Q_s(m/2)^{\oplus t_{p,s,m}}\\
 &\oplus\bigoplus_{\substack{1\le s\le d\\m\ge2\text{ even}}}
     \mathcal T_s(m/2)^{\oplus t_{p-1,s,m}},
\end{split}
\end{equation}
where $t_{-1,s,m}=0$. Moreover,
\begin{align}
 \mathcal Q_s(a)&\simeq
 \Ish^s\mathbin{\times}_{\KM_s/2}\bigl(\KM_s/(2a)\bigr),
 \label{eq:even-Q-pullback}\\
 \mathcal T_s(a)&\simeq
 \Ish^s\mathbin{\times}_{\KM_s/2}\KM_s[2a].
 \label{eq:even-T-pullback}
\end{align}
Here $\KM_s[2a]$ denotes the kernel of multiplication by $2a$.
The map from $\Ish^s$ is reduction modulo $\Ish^{s+1}$ followed by
the Milnor isomorphism; the other maps are reduction modulo $2$,
and inclusion followed by reduction modulo $2$, respectively.
\end{theorem}
\begin{proof}
The formal cancellation of Lemma~\ref{lem:even-reduction} gives,
separately in each positive weight $s$, a chain-homotopy equivalent
complex with differential $hA_s$, where $A_s$ is a bounded complex
of finite free abelian groups. All changes of basis preserve $s$.

A bounded free complex over $\Z$ decomposes by integral changes of
basis into singleton terms and two-term complexes
\[
 \Z\xrightarrow{a}\Z,\qquad a\ge1.
\]
This decomposition can be constructed degree by degree by putting a nonzero differential
in Smith normal form. The preceding and following differentials
are zero on its nonzero diagonal coordinates, since their
compositions with that differential vanish over the torsion-free
ring $\Z$. Each diagonal entry therefore splits off as a two-term
complex, and induction gives the decomposition. We apply the
same changes of basis to $hA_s$. The resulting blocks are singleton
copies of $\KMW_s$ and
\begin{equation}\label{eq:even-scalar-block}
 \KMW_s\xrightarrow{ah}\KMW_s.
\end{equation}
The entries $a=1$ also contribute, since multiplication by
$h$ need not be invertible.

Specialization of the formal cancellation at $h=2$ identifies
$D_s$, up to unit contractible pairs, with the integer complex
$2A_s$. The block with entry $a$ now contributes $\Z/(2a)$ in
its target degree and has zero homology in its source degree.
The singleton terms contribute the free summands. Multiplying Smith
entries by $2$ preserves their divisibility order, so these are
precisely the invariant factors in
\eqref{eq:even-integer-invariants}. This proves their evenness
and identifies all block multiplicities. The cokernel of
\eqref{eq:even-scalar-block} occurs in the target degree and its
kernel in the next degree, proving
\eqref{eq:even-scalar-homology}. The exceptional constant summand
$\Z$ occurs only in degree zero. The decomposition comes from
chain isomorphisms after contractible summands have been removed,
so it involves no extension between the displayed summands.

To identify the individual kernels and cokernels, we use the Milnor--Witt coefficient square
\[
 \KMW_s\simeq\Ish^s\mathbin{\times}_{\KM_s/2}\KM_s
 \qquad\text{\cite[Theorem 1.4]{Fasel}}.
\]
Under this identification, multiplication by $ah$ is
$(u,z)\mapsto(0,2az)$: the Witt image of $h$ is zero and its
rank is two. Its kernel is therefore
\eqref{eq:even-T-pullback}. Since
$\KMW_s\to\KM_s$ is an epimorphism of sheaves, its image is
$\{0\}\times 2a\KM_s$. Taking the quotient therefore gives
\eqref{eq:even-Q-pullback}. All kernels, cokernels, and pullbacks
are taken in the sheaf category. The pullbacks retain the
quadratic extension data inside each scalar summand.
\end{proof}

If $-1$ is a square in $k$, then $h=2$ in $\GW(k)$, so
$\mathcal Q_s(a)=\KMW_s/(2a)$ and
$\mathcal T_s(a)=\KMW_s[2a]$. Over a general field, the hyperbolic
form $h$ retains the quadratic information in
\eqref{eq:even-Q-pullback}--\eqref{eq:even-T-pullback}.

\section{The orthant complex and integral signature lattices}\label{sec:signatures}
Let $\widetilde B$ be an acyclic skew-symmetrizable matrix satisfying \AS, let $\Gamma$ be its mutable graph, and let $X/\R$ be the associated cluster variety. The integer Witt complex computes the positive-degree Chow--Witt groups, including $\CW^2(X)$, which receives the Euler classes of oriented rank-two vector bundles. We now use signatures to describe its comparison with real cohomology, including the integral image of the real cycle class map.

\subsection{Sign labels and the differential}
For $I\in\Ind(\Gamma)$, we define the affine $\F_2$-space
\begin{equation}\label{eq:signlabels}
\Esgn_I=\left\{\varepsilon\in\F_2^{[d]\setminus I}:\
\sum_{a\notin I}\Btil_{ai}\varepsilon_a=1\pmod2
\text{ for all }i\in I\right\}.
\end{equation}
We use the sign bit $0$ for a positive coordinate and $1$ for a negative coordinate. On $\mathcal O_I$, the exchange relation for $i\in I$ is the character equation $\prod_{a\notin I}x_a^{\Btil_{ai}}=-1$. By saturation, this character map is split. It follows that \eqref{eq:signlabels} has $2^{d-2|I|}$ elements; they label the real connected components of $\mathcal O_I$, each of which is contractible.

For a coefficient ring $R$, we put
\[
C^p_{\mathrm{sgn}}(\Btil;R)=\bigoplus_{\substack{I\in\Ind(\Gamma)\\|I|=p}}R^{\Esgn_I}.
\]
If $J=\{j_1<\cdots<j_{p+1}\}$ and $\varepsilon\in\Esgn_J$, let $\varepsilon^{j_a,b}$ be its extension to $[d]\setminus(J\setminus\{j_a\})$ with value $b$ at $j_a$. We define
\begin{equation}\label{eq:sign-differential}
(dc)_{J,\varepsilon}=\sum_{a=1}^{p+1}(-1)^{a-1}
\left(c_{J\setminus\{j_a\},\varepsilon^{j_a,0}}
-c_{J\setminus\{j_a\},\varepsilon^{j_a,1}}\right).
\end{equation}
Both extensions are allowed because there is no edge between vertices of $J$.

\begin{theorem}\label{thm:sign}
The maps \eqref{eq:sign-differential} make $C^\bullet_{\mathrm{sgn}}(\Btil;R)$ a complex and give natural group isomorphisms
\[
H^q(C^\bullet_{\mathrm{sgn}}(\Btil;R))\simeq H^q(X(\R),R).
\]
In particular the real cohomology vanishes above $\alpha(\Gamma)$. The complex is determined by $\Ind(\Gamma)$ and $\Btil\bmod2$.
\end{theorem}
\begin{proof}
To verify $d^2=0$, we fix two vertices to be removed from an independent set. The two possible orders have opposite incidence signs, and extension of sign labels in the two coordinates commutes. The corresponding four terms cancel in pairs.

We compute real cohomology by applying the real-locus functor to $\Omega_\bullet$. A stratum of algebraic codimension $p$ has real codimension $p$, and its ordered normal frame gives an integral Thom orientation. The support spectral sequence has terms
\[
\bigoplus_{|I|=p}H^{q}(\mathcal O_I(\R),R)
\]
with support shift $p$. Every connected component of a stratum is contractible, so only the row $q=0$ is nonzero. A framed normal coordinate has local residue ``positive side minus negative side''. At a higher intersection the ordered product orientation gives $(-1)^{a-1}$. The signs of the remaining coordinates are unchanged. The differential is therefore \eqref{eq:sign-differential}, and the one-row spectral sequence gives the asserted isomorphism.
\end{proof}

\subsection{The exact signature lattice of a torus}
For the integral comparison, we use the convention $\Ish^s=\Wsh$ for $s\leq0$ and the Pfister convention $\pf u=1-\la u\ra$. The normalized signature in degree $s$ is $2^{-s}\sig$; for negative $s$ this means multiplication by $2^{-s}$.

\begin{lemma}\label{lem:lattice}
Let $T=(\Gm^t)_{\R}$ and let $\eta_a$ be the indicator that its $a$th coordinate is negative. Normalized signature is injective on $\Ish^s(T)$ and its image is
\begin{equation}\label{eq:lattice}
\Lambda_s(\F_2^t)=\bigoplus_{A\subseteq[t]}
2^{\max(|A|-s,0)}\Z\prod_{a\in A}\eta_a
\ \subseteq\ \Z^{\F_2^t}.
\end{equation}
The direct sum denotes the lattice basis given by Boolean monomials. In particular,
\begin{equation}\label{eq:lattice-bound}
2^{\max(t-s,0)}\Z^{\F_2^t}\subseteq\Lambda_s(\F_2^t)
\subseteq\Z^{\F_2^t},
\end{equation}
and the lattice is the full function lattice if $s\geq t$. The lattice is invariant under affine changes of the Boolean coordinates.
\end{lemma}
\begin{proof}
For $s\geq0$, successive evaluation at $1$ and contraction split the sections as
\[
\Ish^s(T)=\bigoplus_{A\subseteq[t]}
\pf{t_a:a\in A}\, I^{s-|A|}(\R).
\]
The displayed multiplication identifies each contraction summand with its image. Since $I^u(\R)=2^{\max(u,0)}\Z$, the coefficients are explicit. When $|A|>s$, contraction reaches the Witt group, while the product still lies in $I^{|A|}\subseteq I^s$. The signature of the product of $|A|$ Pfister forms is $2^{|A|}\prod_{a\in A}\eta_a$. After division by $2^s$, the generator has coefficient
$2^{|A|+\max(s-|A|,0)-s}=2^{\max(|A|-s,0)}$. For $s<0$, the sheaf is $\Wsh$ and normalized signature is $2^{-s}$ times its ordinary signature. Multiplying the $s=0$ lattice by this factor therefore gives the same formula.
The Boolean monomials form an integral basis of all functions on $\F_2^t$: their evaluation matrix is triangular with diagonal entries $1$ after ordering subsets. This proves injectivity, the lattice formula, and \eqref{eq:lattice-bound}.

Translations of Boolean coordinates are induced by replacing torus coordinates by their negatives. Every invertible linear change over $\F_2$ lifts to an integral unimodular monomial change of torus coordinates, using elementary matrices and permutations. Since these changes preserve the intrinsic signature image, the lattice is affine invariant.
\end{proof}
A translated split torus has the same description on its affine space of sign components. Thus $\Lambda_s(\Esgn_I)$ is well defined after any choice of affine Boolean coordinates; different choices give the same lattice.

\subsection{An integral lattice model for all \texorpdfstring{$\Ish^j$}{I-j}-cohomology}
We can now describe the integral comparison on the whole cellular complex. For $j\geq0$, set
\begin{equation}\label{eq:I-complex}
C_I^p(X,j)=\bigoplus_{|I|=p}\Lambda_{j-p}(\Esgn_I)
\ \subseteq\ C^p_{\mathrm{sgn}}(\Btil;\Z).
\end{equation}

\begin{theorem}\label{thm:I-model}
The differential \eqref{eq:sign-differential} preserves the lattices \eqref{eq:I-complex}. There are isomorphisms
\begin{equation}\label{eq:I-cohom}
H^q_{\Nis}(X_{\R},\Ish^j)\simeq H^q(C_I^\bullet(X,j)),\qquad j\geq0.
\end{equation}
Under these identifications and the isomorphism of Theorem~\ref{thm:sign},
the inclusion of complexes induces the real cycle class map
\[
\cl^q_j:H^q_{\Nis}(X_{\R},\Ish^j)\longrightarrow H^q(X(\R),\Z).
\]
Thus integral $\Ish^j$-cohomology and its comparison with real cohomology reduce to finite integer linear algebra.
\end{theorem}
\begin{proof}
By Theorem~\ref{thm:acyclic-model}, the cellular complex for $\Ish^j$ has degree-$p$ cochain group
\[
\bigoplus_{|I|=p}H^0(\mathcal O_I,\Ish^{j-p}),
\]
where the normal twist is trivialized by the ordered mutable coordinates. After removing the affine factor by homotopy invariance, Lemma~\ref{lem:lattice} identifies each group injectively with the corresponding lattice in \eqref{eq:I-complex}.

Compatibility of the real cycle map with localization, Thom classes, and pushforward is established in \cite{Jacobson} and \cite[Appendix A, especially A.9 and Corollary A.11]{HWXZ}. We check the normalization explicitly to identify the resulting differential. For a discrete valuation with real normal parameter $u$, a Witt class with signatures $a_+,a_-$ on the two normal sides has second-residue signature $(a_+-a_-)/2$. At a codimension-$p$ term we divide the source signatures by $2^{j-p}$ and the residue signature by $2^{j-p-1}$. Hence the normalized residue is exactly the difference of the two normalized signatures. Ordered normal coordinates give the incidence signs in \eqref{eq:sign-differential}.

The algebraic cellular differential therefore maps the lattice into the next lattice, and injectivity on every term identifies it with the restriction of \eqref{eq:sign-differential}. This proves that the lattices form a complex and gives \eqref{eq:I-cohom}. The same compatibility identifies the induced map with $\cl^q_j$.
\end{proof}

\begin{corollary}\label{cor:real-bounds}
Let $d=\dim X$ and $0\leq i\leq\alpha(\Gamma)$. For every $j\geq0$, the cokernel of $\cl^i_j$ is annihilated by
\begin{equation}\label{eq:cycle-exp}
2^{\max(d-i-j,0)}.
\end{equation}
The map is surjective if $j\geq d-i$. For $i>0$ it is an isomorphism if $j\geq d-i+1$, and for $i=0$ it is an isomorphism if $j\geq d$. Above $\alpha(\Gamma)$ both groups vanish.
\end{corollary}
\begin{proof}
The bounds follow from the inclusions in each cochain degree. In degree $p$, Lemma~\ref{lem:lattice} gives
\[
2^{\max(d-p-j,0)}C^p_{\mathrm{sgn}}\subseteq C_I^p(X,j)\subseteq C^p_{\mathrm{sgn}}.
\]
Given an integral topological cocycle in degree $i$, multiplication by \eqref{eq:cycle-exp} gives an element of the source lattice. It is still a cocycle, since the differential on the subcomplex is the restriction of the same differential. This proves the cokernel assertion and surjectivity range.

For injectivity in positive degree it suffices that $C_I^{i-1}=C^{i-1}_{\mathrm{sgn}}$: a source cocycle which is a target boundary then already has a primitive in the source. This equality holds for $j\geq d-i+1$. In degree zero, termwise injectivity suffices, and surjectivity holds for $j\geq d$. The vanishing assertion follows from the length of the two complexes.
\end{proof}
These bounds depend on the degree because the torus dimension decreases to $d-2p$. Computing the individual lattice matrices can give stronger bounds. Related bounds for affine-cellular and linear schemes appear in \cite{HWXZ,Hennig}.

\begin{corollary}\label{cor:sign-W}
For the real cluster variety of an acyclic seed satisfying \AS,
\[
H^q_{\Nis}(X,\Wsh_\Q)\simeq H^q(X(\R),\Q).
\]
Moreover,
\[
\CW^p(X)_\Q\simeq\CH^p(X)_\Q\oplus H^p(X(\R),\Q).
\]
In particular, if $p>0$, the right side is $H^p(X(\R),\Q)$.
\end{corollary}
\begin{proof}
After tensoring with $\Q$, all the lattices become the full function spaces. This gives the first isomorphism, which also follows from real \'etale comparison. After inverting $2$, the Milnor--Witt pullback square splits into its Milnor and Witt parts; rationally,
$\KMW_p\otimes\Q\simeq\KM_p\otimes\Q\oplus\Wsh_\Q$.
Taking degree-$p$ cohomology and applying Proposition~\ref{prop:acyclic-chow} gives the remaining assertions.
\end{proof}

\subsection{All positive-degree integral Chow--Witt groups}
\begin{theorem}\label{thm:all-CW}
Let $X$ be the real cluster variety of an acyclic seed satisfying \AS. For every $p>0$ there are natural isomorphisms
\begin{equation}\label{eq:all-CW}
\CW^p(X)\xrightarrow{\sim}H^p_{\Nis}(X,\Ish^p)
\xrightarrow{\sim}H^p(C_I^\bullet(X,p)).
\end{equation}
The integral real cycle map is induced by the inclusion $C_I(X,p)\subseteq C_{\mathrm{sgn}}$. Its cokernel is annihilated by $2^{\max(d-2p,0)}$, and its image is contained in $2H^p(X(\R),\Z)$.
\end{theorem}
\begin{proof}
The first isomorphism is \eqref{eq:universal-CW} from
Theorem~\ref{thm:universal-witt}; Theorem~\ref{thm:I-model} gives
the second and its compatibility with the real cycle map.
Corollary~\ref{cor:real-bounds}, with $i=j=p$, gives the cokernel bound.

The real cycle map modulo $2$ factors through the reduction to $\CH^p(X)/2$ and the Borel--Haefliger map, by \cite[Proposition 3.21]{HWXZ}. The reduction vanishes by Chow vanishing. The coefficient sequence $0\to\Z\xrightarrow{2}\Z\to\Z/2\to0$ shows that the kernel of reduction on $H^p(X(\R),\Z)$ is exactly $2H^p(X(\R),\Z)$. This proves the image assertion.
\end{proof}

\begin{corollary}\label{cor:real-middle}
If $d=2a>0$, then $H^a(X(\R),\Z)$ is a finite group of odd order. For every $d>0$, the nonzero real Betti numbers occur only in degrees $r<d/2$.
\end{corollary}
\begin{proof}
Taking $p=a$ in Theorem~\ref{thm:all-CW}, we find that the real cycle map is surjective and its image is contained in $2H^a$. Consequently multiplication by $2$ on this finitely generated abelian group is surjective. It follows that both its free part and its $2$-primary torsion vanish. Condition \AS\ gives $2|I|\le d$ for every independent set $I$, hence
$\alpha(\Gamma)\le\lfloor d/2\rfloor$. Theorem~\ref{thm:sign} therefore gives
vanishing in degrees greater than $a$. For odd $d$, the same bound places
all nonzero real Betti numbers in degrees $r\le\lfloor d/2\rfloor<d/2$.
\end{proof}

\section{Real summands and rational stable motivic types}\label{sec:rational}
The rational plus/minus decomposition separates the ordinary
motive and real realization, as in \cite[Proposition~4.3]{LiuAbelian}.
We will use this separation to combine the motive splitting with the
real cohomology computed in Appendix~\ref{sec:signatures}.

\subsection{The rational stable decomposition}
We use the projectors
\[
e_+=\frac{1+\langle-1\rangle}{2},\qquad
e_-=\frac{1-\langle-1\rangle}{2},\qquad
\one_\pm=e_\pm\one_\Q.
\]
For a rational motivic spectrum $E$, write $E_\pm=e_\pm E$.
The two parts are described by the standard equivalences
\begin{equation}\label{eq:plusminus}
\SH(k)_\Q^+\simeq\DM(k;\Q),\qquad
\SH(k)_\Q^-\simeq D((\Spec k)_{\ret},\Q).
\end{equation}
These are due to Cisinski--D\'eglise \cite[Section 16]{CD} and Bachmann \cite[Theorem 35 and Proposition 41]{Bachmann}. The minus category vanishes for a nonreal field. For $k=\Q$, the real spectrum has a single point, so
\begin{equation}\label{eq:Q-minus}
\SH(\Q)_\Q^-\simeq D(\Q).
\end{equation}
For a smooth variety $X/\Q$, base change to $\R$ identifies the minus image of $\Sigma^\infty X_+$ under \eqref{eq:Q-minus} with the rational chains of $X(\R)$ \cite[Proposition 36 and Corollary 42]{Bachmann}.

\begin{proposition}\label{thm:main-rational}
Let $X_0/\mathbb Q$ be smooth. Suppose that its suspension spectrum is finite cellular and that every $H^j(X_0(\mathbb C),\mathbb Q)$ is a split mixed Tate Hodge structure. Set
\[
b_{j,s}=\dim_{\mathbb Q}\gr^W_{2s}H^j(X_0(\mathbb C),\mathbb Q),\qquad
\beta_r=\dim_{\mathbb Q}H^r(X_0(\mathbb R),\mathbb Q).
\]
For every field extension $k/\mathbb Q$ there is a noncanonical equivalence
\begin{equation}\label{eq:main-split}
\Sigma^\infty_{\mathbb P^1}X_{k,+}\otimes\mathbb Q\simeq
\bigoplus_{j,s}(S^{j,s}\one_{+,k})^{\oplus b_{j,s}}
\oplus\bigoplus_r(S^{r,0}\one_{-,k})^{\oplus\beta_r}.
\end{equation}
Consequently, for every rational motivic spectrum $E$,
\begin{equation}\label{eq:main-coeff}
E^{a,b}(X_k)\simeq
\bigoplus_{j,s}E_+^{a-j,b-s}(k)^{\oplus b_{j,s}}
\oplus\bigoplus_r E_-^{a-r,b}(k)^{\oplus\beta_r}.
\end{equation}
\end{proposition}
\begin{proof}
We first determine the plus part. Finite cellularity makes $M(X_0)$ mixed Tate, and smooth duality gives the same property for $M^c(X_0)$. By duality, the assumed splitting of the mixed Hodge structures on ordinary cohomology gives the corresponding splitting on Borel--Moore homology. The splitting criterion \cite[Theorem~4.28]{Hlavinka}, followed again by smooth duality, therefore gives
\[
M(X_0)\simeq\bigoplus_{j,s}\Q(s)[j]^{\oplus b_{j,s}}.
\]
The first equivalence in \eqref{eq:plusminus} identifies its summands with $S^{j,s}\one_+$. For the minus part, the rational chains of the real locus split into their homology. Thus \eqref{eq:Q-minus} gives the displayed minus summands. Base change preserves these equivalences and the projectors, proving \eqref{eq:main-split}. Applying $\Hom(-,\Sigma^{a,b}E)$ proves \eqref{eq:main-coeff}.
\end{proof}
For an acyclic cluster variety $X_0/\Q$ satisfying \AS,
Theorem~\ref{thm:acyclic-model} supplies finite cellularity. Taking
motives shows that $M(X_0)$ and, by smooth duality, $M^c(X_0)$
are mixed Tate. Theorem~\ref{thm:gysin} computes the multiplicities as
\[
 b_{j,s}=\dim_{\Q}H^{j-s}(G_{\mathrm{Gys}}^{\bullet,s}).
\]
If the seed is RFR, its Hodge cohomology is split by
\cite[Theorem~8.3]{LS1} in the skew-symmetric case and by
\cite[Proposition~4.4]{Hlavinka} in general. Applying the preceding
proof and then base change gives, for every field extension $k/\Q$,
\[
 M(X_k)\simeq\bigoplus_{j,s}\Q_k(s)[j]^{\oplus b_{j,s}}.
\]
Here $\Q_k(s)$ denotes the Tate motive over $k$. The cellular
filtration and its oriented differential therefore determine the
mixed Tate property and the multiplicities, while Hodge splitting
gives the direct-sum decomposition. The orthant complex of
Appendix~\ref{sec:signatures} computes the minus multiplicities in
\eqref{eq:main-split}.

For the full RFR sink-recurrent class over $\Q$, the same rational
stable decomposition follows from the motive splitting in
\cite[Theorem~4.29]{Hlavinka} and the preceding real-realization
argument.

An integral splitting also exists in the elementary rank-one case
$X_1=\Spec k[x,x',t^{\pm1}]/(xx'-1-t)$:
\[
 \Sigma^\infty X_{1,+}\simeq\one\oplus S^{1,1}\oplus S^{2,2}.
\]
Indeed, the complement of $V(x)\cong\A^1$ is the torus with coordinates
$(x,t)$. Rotating its purity triangle gives the attaching map
$S^{1,1}\to\Sigma^\infty(\Gm^2)_+$.
Its component in the $S^{1,1}$ summand corresponding to $x$ is the
identity: compare with the purity triangle for $0\subset\A^1$ using
$\A^1\to X_1$, $u\mapsto(u,0,-1)$, and naturality of purity.
A triangular automorphism therefore cancels that summand and gives
the displayed cofiber.

\end{document}